\documentclass[11pt,a4paper]{amsart}
\usepackage{amsmath, amssymb}
\usepackage{amsfonts}
\usepackage{amsthm}
\usepackage{amscd}
\usepackage{hyperref} 
\usepackage[dvipsnames]{xcolor}
\allowdisplaybreaks

\newtheorem{theorem}{Theorem}[section]
\newtheorem{lemma}{Lemma}[section]
\newtheorem{proposition}{Proposition}[section]
\newtheorem{corollary}{Corollary}[section]

\theoremstyle{definition}
\newtheorem{definition}{Defintion}[section]
\newtheorem{notation}[definition]{Notation}

\theoremstyle{remark}
\newtheorem{remark}{Remark}[section]

\numberwithin{equation}{subsection}
\newcount\colveccount
\newcommand*\colvec[1]{
	\global\colveccount#1
	\begin{pmatrix}
		\colvecnext
	}
	\def\colvecnext#1{
		#1
		\global\advance\colveccount-1
		\ifnum\colveccount>0
		\\
		\expandafter\colvecnext
		\else
	\end{pmatrix}
	\fi
}

\newcommand{\Go}{\mathsf{G}_0}
\newcommand{\G}{\mathsf{G}}
\newcommand{\Ga}{\mathsf{G}_a}
\newcommand{\Gl}{\mathsf{SL}}

\newcommand{\sL}{\mathsf{L}}
\newcommand{\sZ}{\mathsf{Z}}

\newcommand{\fg}{\mathfrak{g}}
\newcommand{\sP}{\mathsf{P}}
\newcommand{\sQ}{\mathsf{Q}}
\newcommand{\sT}{\mathsf{T}}
\newcommand{\sH}{\mathsf{H}}

\newcommand{\cX}{\mathcal{X}}
\newcommand{\cY}{\mathcal{Y}}

\newcommand{\cC}{\mathcal{C}}

\newcommand{\R}{\mathbb{R}}

\newcommand{\s}{\mathsf{Stab}}
\newcommand{\sHom}{\mathsf{Hom}}
\newcommand{\oW}{\overrightarrow{W}}
\newcommand{\oV}{\overrightarrow{V}}

\newcommand{\oR}{\overrightarrow{\mathbb{R}}}
\newcommand{\defeq}{\mathrel{\mathop:}=}

\newcommand{\flow}{\mathsf{U}\Gamma}
\newcommand{\cflow}{\widetilde{\mathsf{U}\Gamma}}
\newcommand{\bdry}{\partial_\infty\Gamma}
\newcommand{\bH}{\mathbb{H}^{n,n-1}}
\newcommand{\bR}{\mathbb{R}^{n,n-1}}
\newcommand{\sad}{\mathsf{Ad}}

\title{Avatars of Margulis invariant and Proper actions}
\author{Sourav Ghosh}
\address{Ashoka University}
\email{sourav.ghosh@ashoka.edu.in, sourav.ghosh.bagui@gmail.com}
\date{\today}

\thanks{The author acknowledge support from DFG SPP 2026 grant, OPEN/16/11405402 grant and the Ashoka University annual research grant.}
\begin{document}
	\begin{abstract}
		In this article, we provide a necessary and sufficient criterion for proper actions on $\mathbb{H}^{n,n-1}$ in terms of certain special Anosov representations in $\mathsf{SO}(n,n)$. Moreover, we show that affine Anosov representations of any word hyperbolic group in $\mathsf{SO}(n,n-1)\ltimes\mathbb{R}^{2n-1}$ are infinitesimal versions of such special Anosov representations. Finally, using the above two results we interpret Margulis spacetimes as infinitesimal versions of quotient manifolds of $\mathbb{H}^{n,n-1}$.
		
		In the appendix, we give a description of the appropriate cross-ratios in our setting and their infinitesimal versions.
	\end{abstract}
	
	\maketitle
	\tableofcontents

	\pagebreak

	\section{Introduction}
	
	The study of tilings gives rise to the study of proper actions. In a celebrated result, Bieberbach classified the symmetries of `crystals'. Any subgroup of $\mathsf{O}(n,\R)\ltimes\R^n$ whose action on $\R^n$ is proper and cocompact is called a \emph{crystallographic} group. Bieberbach \cite{B1,B2} (see also \cite{Buser}) showed that for each fixed $n$ there are only finitely many isomorphism classes of $n$-dimensional crystallographic groups and they contain a normal subgroup of finite index isomorphic to $\mathbb{Z}^{n}$. Later, Auslander and Markus \cite{AM} constructed examples of subgroups of $\mathsf{Aff}(n,\R):=\mathsf{GL}(n,\R)\ltimes\R^n$ whose action on $\R^n$ are proper and cocompact and which do not have a normal subgroup of finite index isomorphic to $\mathbb{Z}^n$. The examples they constructed were normal subgroups of finite index isomorphic to $\mathbb{Z}^{\ltimes n}$. Auslander \cite{Aus,Aus2} attempted to show that any subgroup of $\mathsf{O}(n,\R)\ltimes\R^n$ whose action on $\R^n$ is proper and cocompact is virtually polycyclic and failed. Later, the statement was rechristened as \emph{Auslander conjecture} by Fried--Goldman \cite{FG}. This conjecture is still unsolved for the general case but has been shown to hold true in dimension less than $7$ by the work of Fried--Goldman \cite{FG} and Abels--Margulis--Soifer \cite{AMS2}. Recently, a generalization of the Auslander conjecture for homogeneous spaces appeared in \cite{Tom} with partial resolution. In fact, the study of proper action of discrete groups on homogeneous spaces can be traced back to Borel \cite{Borel}. He studied compact quotients of symmetric spaces. Later, the study of pseudo-Riemannian homogeneous spaces was pioneered by Kulkarni \cite{Kul}, Kobayashi \cite{Kob}, Benoist \cite{Ben} and Okuda \cite{Oku}. They found that the signature of a homogeneous space acts as an obstruction for the existence of proper actions of a discrete group on it (please check \cite{Kob3} for a recent survey). Let $g\in\mathsf{GL}(n,\R)$, $v\in\R^n$ and $(g,v)\in\mathsf{Aff}(n,\R)$. Then we call $\mathsf{GL}(n,\R)$ (resp. $g$) the \emph{linear} part of $\mathsf{Aff}(n,\R)$ (resp. $(g,v)$) and $\R^n$ (resp. $v$) the \emph{translational} part of $\mathsf{Aff}(n,\R)$ (resp. $(g,v)$). We observe that the map $L$ sending any element to its linear part is a homomorphism and for any subgroup $\sH$ of $\mathsf{Aff}(n,\R)$ we call $L(\sH)$ the linear part of $\sH$.
	
	Meanwhile, Margulis \cite{Margulis1,Margulis2} answered a question asked by Milnor \cite{Milnor} related to the Auslander conjecture in the negative. He showed that the conjecture would fail if one drops the cocompactness assumption. He constructed actions of non-abelian free subgroups of $\mathsf{GL}(3,\R)\ltimes\R^3$ which act properly on $\R^3$. It was known due to prior works of Kostant--Sullivan \cite{KS} and Fried--Goldman \cite{FG} that the linear part of the Zariski closure of such a group has to be $\mathsf{SO}(2,1)$. The quotient manifolds obtained from $\R^3$ under the proper action of a non-abelian free subgroup of $\mathsf{SO}(2,1)\ltimes\R^3$ are called \emph{Margulis spacetimes}. Furthermore, Drumm \cite{Drumm2} classified the linear holonomy of Margulis spacetimes and constructed \cite{Drumm1} nice fundamental domains for a large class of them. It was conjectured that any Margulis spacetime admits a fundamental domain of the type constructed by Drumm. This conjecture is known as the \emph{Crooked plane conjecture}. Recently, the Crooked plane conjecture was resolved by Danciger--Gu\' eritaud--Kassel \cite{DGK2}.  While constructing Margulis spacetimes, Margulis introduced certain invariants, which later came to be known as the Margulis invariants, to give a necessary criterion for proper actions. Later, Labourie \cite{Labourie2} introduced continuous versions of the Margulis invariants. We call these invariants as Labourie-Margulis invariants. Subsequently, Goldman--Labourie--Margulis \cite{GLM} used the Labourie-Margulis invariants to give a necessary and sufficient condition for a representation to be a Margulis spacetime. Classification results of similar nature were also independently obtained by Danciger--Gu\' eritaud--Kassel \cite{DGK1} using different techniques. In fact, using those techniques they went on to prove a twenty year old conjecture of Drumm-Goldman. They demonstrated that Margulis spacetimes are tame, i.e. Margulis spacetimes are homeomorphic to the interior of a compact manifold with boundary. Tameness of Margulis spacetimes was also independently proved by Choi--Goldman \cite{CG}. Recently, similar results for Margulis spacetimes with parabolics was obtained by Choi--Drumm--Goldman \cite{ChoiDG}. The construction of Margulis spacetimes and Margulis invariants were generalized by Abels--Margulis--Soifer \cite{AMS}. For $n$ even they constructed more examples of non-abelian free subgroups of $\mathsf{GL}(2n-1,\R)\ltimes\R^{2n-1}$ whose Zariski closure is $\mathsf{SO}(n,n-1)\ltimes\R^{2n-1}$ and which act properly on $\R^{2n-1}$. They also proved the nonexistence of any such subgroups for odd $n$. Suppose $\Gamma$ is a finitely generated hyperbolic group. We abuse notation and call an injective homomorphism $(\rho,u):\Gamma\to\mathsf{SO}(n,n-1)\ltimes\R^{2n-1}$, a Margulis spacetime if $(\rho,u)(\Gamma)\backslash\R^{2n-1}$ is a manifold. The constructions of \cite{AMS} were further generalized by Smilga in \cite{Smilga,Smilga4}. Generalizing previous works of Kim \cite{Kim} and Ghosh \cite{Ghosh5}, recently it was proved in \cite{Ghosh6} that only finitely many well chosen Margulis invariants are enough to determine a conjugacy class of Margulis spacetimes. It might appear from the above discussion that free groups are the only groups which admit proper affine actions but that is not the case. First examples of proper affine actions of right-angled Coxeter groups was obtained by Danciger--Gu\' eritaud--Kassel \cite{DGK3}. It remains to be seen if there are other classes of groups which admit proper actions or not.
	
	The classification results of Goldman--Labourie--Margulis \cite{GLM} and Goldman--Labourie \cite{GL} point towards a description of Margulis spacetimes in terms of Anosov representations. This project is still ongoing with affirmative results by Ghosh \cite{Ghosh1} and Ghosh--Treib \cite{GT}. Anosov representations generalize the notion of convex cocompactness for higher rank Lie groups \cite{GGKW}. They were introduced by Labourie \cite{Labourie} to provide a geometric characterization of certain special representations of surface groups called \emph{Hitchin representations}. Hitchin representations can be defined as those representations which lie in the components of the representation variety containing Fuchsian representations of a surface group. One interesting thing about Hitchin representations is that their moduli space is topologically trivial and it is the same as the moduli of solutions of the Yang--Mills equations under certain symmetry conditions \cite{Hit}. Interestingly, Hitchin representations also make their appearance in the study of Margulis spacetimes through the works of Danciger--Zhang \cite{DZ} and Labourie \cite{Labourie3}. Generalizing previous results of Goldman--Margulis \cite{GM}, Mess \cite{Mess} and Labourie \cite{Labourie2}, they showed that Hitchin representations in $\mathsf{SO}(n,n-1)$ do not admit affine deformations which act properly on $\R^{2n-1}$. The definition of an Anosov representation given by Labourie is dynamical in nature and its dynamics resembles the notion of an Axiom A flow appearing in the dynamical systems literature \cite{Bowen, ru}. Later on, Guichard--Wienhard \cite{GW2} extended the notion of an Anosov representation to include representations of any finitely generated hyperbolic group into semisimple Lie groups. They also established relations between Anosov representations and proper actions on homogeneous spaces. Subsequently, a more algebraic description of Anosov representations, in terms of uniform gaps in singular value or eigenvalue spectra, and relation between Anosov representations and proper actions on homogeneous spaces appeared in the works of Kapovich--Leeb--Porti \cite{KLP,KLP1,KLP2}, Gu\'eritaud--Guichard--Kassel--Wienhard \cite{GGKW}, Bochi--Potrie--Sambarino \cite{BPS} and Kassel--Potrie \cite{KP}. In this article, we stick to the dynamical description of an Anosov representation. Further dynamical properties of these representations were proved by Bridgeman--Canary--Labourie--Sambarino in \cite{BCLS}. The appropriate version of Anosov representations of a hyperbolic group into affine Lie groups of the kind $\mathsf{SO}(n,n-1)\ltimes\R^{2n-1}$, was introduced in the works of Ghosh \cite{Ghosh1,Ghosh2} and Ghosh--Treib \cite{GT}. They showed that Margulis spacetimes can be characterized by affine Anosov representations. We prove a similar characterization relating proper actions of hyperbolic groups on the homogeneous space $\mathbb{H}^{n,n-1}$ and certain special Anosov representations (this generalization provide an alternate proof of a statement from \cite{DZ} communicated to us by Danciger and Zhang). Let $\Gamma$ be a finitely generated hyperbolic group. We denote the inner product on $\R^{2n}$ whose symmetry group is $\mathsf{SO}(n,n)$ by $\langle\mid\rangle$ and the connected component containing identity of $\mathsf{SO}(n,n)$ by $\G$. Let $\R^{n,n}=\R^{n,n-1}\oplus\R e_{2n}$ and $\Go$ be the group of those elements in $\G$ which fix the vector $e_{2n}$. Let $\bH:=\G e_{2n}\subset\R^{2n}$ and $\Ga:=\G_0\ltimes\R^{2n-1}$. We observe that $\bH\cong\G/\G_0$. We call a subspace of $\R^{2n-1}$ a \emph{null} subspace if its orthogonal is a maximal isotropic subspace of $\R^{n,n-1}$. We note that maximal isotropic subspaces of $\R^{n,n-1}$ are $(n-1)$-dimensional and null subspaces of $\R^{n,n-1}$ are $n$-dimensional. Let $\sP_0^\pm$ be the stabilizer inside $\G_0$ of two transverse $n$-dimensional null subspaces of $\R^{n,n-1}$, $\sP_a^\pm$ be the stabilizer inside $\Ga$ of two transverse $n$-dimensional null subspaces of $\R^{n,n-1}$ and $\sP^\pm$ be the stabilizer inside $\G$ of two $(n-1)$-dimensional isotropic subspaces of $\R^{n,n}$ which are oriented and whose orthogonal subspaces are transverse to each other (for more details see Sections \ref{sec:proto-neutral} and \ref{sec:aar}). Moreover, we denote $\sHom(\Gamma,\G_0,\sP_0^\pm)$ (resp. $\sHom(\Gamma,\G,\sP^\pm)$) to be the space of all injective homomorphisms of $\Gamma$ inside $\G_0$ (resp. $\G$) which are Anosov with respect to $\sP_0^\pm$ (resp. $\sP^\pm$) and we denote $\sHom(\Gamma,\Ga,\sP_a^\pm)$ to be the space of all injective homomorphisms of $\Gamma$ inside $\Ga$ which are affine Anosov with respect to $\sP_a^\pm$. We prove the following:
	\begin{theorem}\label{thm1} (see Theorem \ref{thm:p})
		We recall that $\G$ is the connected component of $\mathsf{SO}(n,n)$ containing identity. Suppose $\rho$ is a representation of $\Gamma$ in $\G$ which is Anosov with respect to the stabilizer of an oriented $(n-1)$-dimensional isotropic subspace of $\R^{n,n}$, that is, $\rho\in\sHom(\Gamma,\G,\sP^\pm)$. Then the action of $\rho(\Gamma)$ on $\bH$ is proper if and only if $\rho$ is Anosov in $\mathsf{SL}(2n,\R)$ with respect to the stabilizer of an $n$-dimensional subspace. 
	\end{theorem}
	Goldman--Margulis \cite{GM} kicked off the study of Margulis spacetimes via deformation of related objects. They showed that the marked Margulis invariant spectrum of an affine representation in dimension three can be interpreted as derivatives of the marked length spectrum of surfaces. Deformation techniques were also used by Danciger--Gu\' eritaud--Kassel \cite{DGK1} to show that three dimensional Margulis spacetimes are the rescaled limits of collapsing AdS spacetimes. In fact, their proof of the Crooked plane conjecture \cite{DGK3} also used deformation techniques but in the context of arc complexes. A general framework to describe transitions between geometries which are sub-geometries of a larger ambient geometry was developed by Cooper--Danciger--Wienhard \cite{CDW}. Furthermore, in a different context Kassel \cite{Ka} and Gu\' eritaud--Kassel \cite{GK} showed that small deformations of proper actions still give rise to proper actions. In this article, we generalize the derivative interpretation of the Margulis invariants and show that Margulis invariant spectra can be obtained as the derivative of the middle eigenvalue gap spectra of representations in $\G$. Let $h\in\Go$ be \emph{pseudo-hyperbolic} i.e. the unit eigenspace of $h$ is one dimensional and $h$ does not have $-1$ as an eigenvalue. Let $W^h_\pm$ respectively be the subspaces of $\R^{2n-1}$ on which the action of $h$ is contracting or expanding. There is a consistent way of choosing a direction along the unique eigenspace of $h$ with eigenvalue $1$ (for more details please see Section \ref{sec:aar}). Let $v^h_0$ be the unique eigenvector of $h$ with eigenvalue $1$ such that $\langle v^h_0\mid v^h_0\rangle=1$ and which is positively oriented with respect to the aforementioned choice of direction. In particular, when $n=2$, the subspaces $W^h_\pm$ are one dimensional and light-like. The space of light-like vectors is a double cone. The choice in this case is done as follows: we choose non zero vectors $v^h_\pm\in W^h_\pm$ which lie in the upper cone and choose $v_0^h$ in such a way that $(v^h_-,v_0^h,v^h_+)$ is positively oriented with respect to the standard orientation on $\R^3$. Then the \emph{Margulis invariant} of an element $g=(h,u)\in\Ga$ is defined as
	\[\alpha(h,u):=\langle u\mid v^h_0\rangle.\]
	We observe that the dimensions of $W_\pm^h$ are $(n-1)$. Let $h_t\in\G$ be an analytic one parameter family with $h=h_0$. Then, for $t$ small enough we obtain a pair of attracting and repelling subspaces of $h_t$ denoted by $W_\pm^{h_t}$ and whose dimensions are also $(n-1)$. Moreover, there is a consistent way of choosing a maximal isotropic subspace $V_+^{h_t}$ inside $(W_+^{h_t})^\perp\subset\R^{2n}$. It is the maximal isotropic subspace of $(W_+^{h_t})^\perp$ which is the deformation of $V_+^{h}=W_+^{h}\oplus\R(v_0^h+e_{2n})\subset(W_+^{h})^\perp$ (for more details please see Sections \ref{sec:proto-neutral} and \ref{sec:aar}). We observe that $(W_-^{h})^\perp\cap V_+^{h}=\R(v_0^h+e_{2n})$ and that $h_t$ preserves the line $(W_-^{h_t})^\perp\cap V_+^{h_t}$. Let $\lambda(h_t)$ be the eigenvalue of the action of $h_t$ on this line. We prove the following:
	\begin{lemma}\label{lem}
		Suppose $g\in\G_0$ is a pseudo-hyperbolic element and $g_t\in\G$ is an analytic one parameter family whose tangent direction at $g=g_0$ is $G$ and $Ge_{2n}=v$. Then 
		\[\alpha(g,v)=\left.\frac{d}{dt}\right|_{t=0}\lambda(g_t).\]
	\end{lemma}
	We give a similar interpretation of Labourie-Margulis invariant too and use it to show that the notion of an affine Anosov representation detect the infinitesimal versions of certain special Anosov representations in $\G$:
	\begin{theorem}\label{thm2} (see Theorem \ref{thm:affano})
		We recall that $\R^{n,n}=\R^{n,n-1}\oplus\R e_{2n}$, $\G$ is the connected component of $\mathsf{SO}(n,n)$ containing identity, $\G_0$ is the subgroup of $\G$ which fixes  $e_{2n}$ and $\Ga=\G_0\ltimes\R^{n,n-1}$. 
		
		Let $\{\rho_t\}_{t\in(-1,1)}$ be an analytic one parameter family of representations of $\Gamma$ in $\G$ with $\rho_0(\Gamma)\subset\Go$. Let ${U}$ be the tangent vector to $\{\rho_t\}_{t\in(-1,1)}$ at $\rho=\rho_0$ and ${u}={U}e_{2n}$. Suppose $(\rho,u)$ is affine Anosov in $\Ga$ with respect to the stabilizer of an $n$-dimensional null subspace of $\R^{n,n-1}$, that is, $(\rho,u)\in\sHom(\Gamma,\Ga,\sP^\pm_a)$. Then there exists $\epsilon>0$ such that for all $t$ with $|t|\in(0,\epsilon)$, $\rho_t$ is Anosov in $\mathsf{SL}(2n,\R)$ with respect to the stabilizer of an $n$-dimensional subspace.
	\end{theorem}
	Finally, combining these two results and using the characterization of Margulis spacetimes in terms of affine Anosov representations, we show that Margulis spacetimes of dimension $(2n-1)$ are infinitesimal versions of manifolds obtained from quotients of the homogeneous space $\mathbb{H}^{n,n-1}$:
	\begin{corollary}\label{cor1}
		Let $\{\rho_t\}_{t\in(-1,1)}$ be an analytic one parameter family of representations of $\Gamma$ in $\G$ with $\rho=\rho_0$ Anosov in $\Go$ with respect to the stabilizer of an $n$-dimensional null subspace of $\R^{n,n-1}$. Let ${U}$ be the tangent vector to $\{\rho_t\}_{t\in(-1,1)}$ at $\rho=\rho_0$ and ${u}={U}e_{2n}$. Suppose $(\rho,u)$ is a Margulis spacetime. Then there exists $\epsilon>0$ such that for all $t$ with $|t|\in(0,\epsilon)$, $\rho_t(\Gamma)$ acts properly on $\bH$.
	\end{corollary}
	Moreover, using the work of Abels--Margulis--Soifer \cite{AMS} and our main result we provide an alternate proof of the following fact which was first obtained by Benoist \cite{Ben}: 
	\begin{corollary}\label{cor2}
		Suppose $n$ is even. Then there exists a non-abelian free subgroup with finitely many generators inside $\G$ which act properly on $\bH$.
	\end{corollary}
	In the appendix, we also generalize results by Charette--Drumm \cite{CD} and Ghosh \cite{Ghosh1}. We define affine cross ratios $\beta$ for any four mutually transverse affine null subspaces in $\R^{n,n-1}$ and show the following:
	\begin{proposition}\label{prop1}
		Suppose $n$ is even, $\rho\in\sHom(\Gamma,\Go,\sP_0^\pm)$ and $(\rho,u)\in\sHom(\Gamma,\Ga)$. Suppose $\gamma,\eta\in\Gamma$ are two infinite order elements such that the four points $\gamma^\pm,\eta^\pm\in\bdry$ are distinct and the sequence $\{\gamma^m\eta^m\}_{m\in\mathbb{N}}\subset\Gamma$ contains a subsequence $\{\gamma^{n_i}\eta^{n_i}\}_{i\in\mathbb{N}}$ consisting only of infinite order elements. Then the following identity holds:
		\[\lim_{i\to\infty}(\alpha(\gamma^{n_i}\eta^{n_i})-\alpha(\gamma^{n_i})-\alpha(\eta^{n_i}))=\beta(\eta^-, \gamma^-,\gamma^+,\eta^+).\]
	\end{proposition}
	Moreover, we also define the linear counterparts $\theta$ of these affine cross ratios defined for any four mutually transverse $(n-1)$-dimensional isotropic subspaces in $\R^{n,n}$ and show that
	\begin{proposition}\label{prop2}
		Suppose $n$ is even, $\rho\in\sHom(\Gamma,\G,\sP^\pm)$ and $\gamma,\eta\in\Gamma$ are two infinite order elements such that the four points $\gamma^\pm,\eta^\pm\in\bdry$ are distinct and the sequence $\{\gamma^m\eta^m\}_{m\in\mathbb{N}}\subset\Gamma$ contains a subsequence $\{\gamma^{n_i}\eta^{n_i}\}_{i\in\mathbb{N}}$ consisting only of infinite order elements. Then the following identity holds:
		\[\lim_{i\to\infty}\frac{\lambda(\gamma^{n_i}\eta^{n_i})^2}{\lambda(\gamma^{n_i})^2\lambda(\eta^{n_i})^2}= \theta(\eta^-, \gamma^-,\gamma^+,\eta^+)^2.\]
	\end{proposition}
	
	\begin{remark}
		Lastly, we would like to mention that Danciger--Zhang has announced independent work in \cite{DZ} which has overlap with some of our results. In particular, Theorem \ref{thm1}, Lemma \ref{lem} and Theorem \ref{thm2} of this article when applied to fundamental groups of compact surfaces without boundary, are respectively similar to Lemma 8.2, Theorem 8.8 and Theorem 6.1 of \cite{DZ}. On the other hand, the results about cross ratios contained in Sections \ref{sec:acr} and \ref{sec:cr} are not obtained by Danciger--Zhang. We would also like to note that, even though Corollary \ref{cor1} has not been stated as a result in \cite{DZ}, for the case of fundamental groups of compact surfaces without boundary, it can also be obtained by jointly applying Theorems 8.8 and 6.1 of \cite{DZ}. In \cite{DZ}, Danciger--Zhang use Lemma 8.2, Theorem 8.8, Theorem 6.1 and they also use certain properties special to Hitchin representations obtained from the works of Labourie \cite{Labourie} and Fock--Goncharov \cite{FoG} to generalize Theorem 1.1 of \cite{Labourie2} and conclude that representations in $\mathsf{PSL(2n-1,\R)\ltimes\R^{2n-1}}$ whose linear parts are Hitchin do not admit proper affine actions on $\R^{2n-1}$.
	\end{remark}

	\section*{Acknowledgements} 
	
	I acknowledge the independent results obtained in \cite{DZ} and I would like to thank Dr. Tengren Zhang and Dr. Jeffrey Danciger for the statement of Theorem\ref{thm:p}. I would also like to thank Prof. Fran\c cois Labourie, Prof. Anna Wienhard, Dr. Nicolaus Tholozan, Dr. Binbin Xu and Dr. Zhe Sun for helpful discussions.  
	
	Moreover, I want to express my sincerest regards to Ms. Saumya Shukla for helping me achieve a smooth simultaneous arXiv publication with \cite{DZ}. Unfortunately, during the revision of this manuscript, she died an untimely death fighting for justice. I miss her terribly and dedicate this work to her indomitable spirit.

	\section{Preliminaries}
	
	In this section we introduce certain preliminary notions and results needed to establish our results.
	
	\subsection{Anosov representations}
	
	In this subsection we define the notion of an Anosov representation and mention some important properties of Anosov representations which will be used later on. Anosov representations into $\mathsf{SL}(n,\R)$ were introduced by Labourie in \cite{Labourie} to show that Hitchin representations satisfy certain nice geometric properties. Later on, Guichard--Wienhard \cite{GW2} extended the notion of an Anosov representation to representations of any hyperbolic group into a semisimple Lie group. Recently, Kapovich--Leeb--Porti gave a different algebraic characterization of Anosov representations in \cite{KLP} and \cite{KLP1}. In this article, we use the dynamical definition of an Anosov representation from the work of Labourie \cite{Labourie} and Guichard--Wienhard \cite{GW2}.
	
	We start by defining the Gromov flow space. It plays a very central role in the dynamical definition of an Anosov representation. Let $\Gamma$ be a finitely generated word hyperbolic group, $\bdry$ be its boundary at infinity and let 
	\[\bdry^{(2)}\defeq\{(p_+,p_-)\mid p_\pm\in\bdry,p_+\neq p_-\}.\] Gromov \cite{Gromov} (see also Champetier \cite{Champetier} and Mineyev \cite{Mineyev}) constructed a cocompact, proper action of $\Gamma$ on $\cflow\defeq\bdry^{(2)}\times\R$, which commutes with the flow:
	\begin{align*}
		\phi_t:\cflow&\rightarrow\cflow\\
		p\defeq(p_+,p_-,p_0)&\mapsto (p_+,p_-,p_0+t)
	\end{align*}
	and whose restriction on $\bdry^{(2)}$ is the diagonal action coming from the natural action of $\Gamma$ on its boundary $\bdry$. Moreover, there exists a metric on $\cflow$ well defined up to H\"older  equivalence such that the $\Gamma$ action is isometric, the flow $\phi_t$ acts by Lipschitz homeomorphisms and every orbit of the flow $\{\phi_t\}_{t\in\mathbb{R}}$ gives a quasi-isometric embedding. The resulting quotient space denoted by $\flow$ is called the \textit{Gromov flow space}. We note that the Gromov flow space is connected and it admits partitions of unity (for more details please see \cite{GT}).
	
	The other important ingredients in the definition of Anosov representations are parabolic subgroups (for a detailed exposition on parabolic subgroups please see Section 3.2 of \cite{GW2}). In this article, to prove certain results we also work with stabilizers of oriented subspaces. These groups strictly speaking are not necessarily parabolic subgroups but they are subgroups of finite index inside parabolic subgroups. We call such subgroups \emph{virtually parabolic} subgroups. Hence, although the theory of Anosov representations due to Guichard-Wienhard \cite{GW2} does not directly apply to these cases, the original theory due to Labourie \cite{Labourie} does. Let $\sH$ be a connected semisimple Lie group. Moreover, let $\sP_\pm$ be a pair of virtually parabolic subgroups of $\sH$ such that their respective overgroups which are parabolic are opposite and $\sP_+\cap\sP_-$ is a finite index subgroup of the intersection of the two opposite parabolic overgroups. We call such a pair of virtually parabolic subgroups as \emph{opposite}. Let $\cX\subset\sH/\sP_+\times\sH/\sP_-$ be the space of  all pairs $(h\sP_+,h\sP_-)$ for $h\in\sH$. We consider the left action of $\sH$ on $\sH/\sP_+\times\sH/\sP_-$ and observe that the action is transitive on $\cX$ and the stabilizer of the point $(\sP_+,\sP_-)\in\cX$ is $\sP_+\cap\sP_-$. Hence $\sH/(\sP_+\cap\sP_-)\cong\cX$.  Moreover, $\cX$ is open in $\sH/\sP_+\times\sH/\sP_-$. Therefore,
	\[\sT_{\left(h\sP_+,h\sP_-\right)}\cX=\sT_{h\sP_+}\sH/\sP_+\oplus\sT_{h\sP_-}\sH/\sP_-.\]
	\begin{definition}
		Let $\Gamma$ be a hyperbolic group and let $\sH$ be a semisimple Lie group with a pair of opposite virtually parabolic subgroups $\sP_\pm$. Then any representation $\rho:\Gamma\to\sH$ is called $\sP_\pm$-Anosov if and only if 
		\begin{enumerate}
			\item There exist continuous, injective, $\rho(\Gamma)$-equivariant limit maps
			\[\xi^\pm:\bdry\rightarrow\sH/\sP_\pm\]
			such that $\xi(p)\defeq(\xi^+(p_+),\xi^-(p_-))\in\cX$ for any $p\in\cflow$ .
			\item There exist positive constants $C,c$ and a continuous collection of $\rho(\Gamma)$-equivariant Euclidean metrics $\|\cdot\|_p$ on $\sT_{\xi(p)}\cX$ for $p\in\cflow$ such that 
			\begin{align*}
				\|v^\pm\|_{\phi_{\pm t}p}\leqslant Ce^{-ct}\|v^\pm\|_p
			\end{align*}
			for all $v^\pm\in\sT_{\xi^\pm(p_\pm)}\sH/\sP_\pm$ and for all $t\geqslant 0$.
		\end{enumerate}
	\end{definition}
	
	\begin{notation}
		We denote the space of all representations $\rho:\Gamma\to\G$ by $\sHom(\Gamma,\G)$, the space of $\sP_\pm$-Anosov representations in $\G$ by $\sHom(\Gamma,\G,\sP_\pm)$ and the space of all $\alpha$-H\"older maps from $\bdry$ to $\sH/\sP_\pm$ by $\cC^\alpha(\bdry,\sH/\sP_\pm)$.
	\end{notation}
	
	Now we state a few theorems which will be important for us later on.
	\begin{theorem}[Labourie \cite{Labourie},Guichard--Wienhard \cite{GW2}]\label{thm:LGW2}
		Suppose $\rho\in\sHom(\Gamma,\G,\sP_\pm)$. Then there exists an open neighborhood $U\subset\sHom(\Gamma,\sH)$ with $\rho\in U$ such that $U\subset\sHom(\Gamma,\G,\sP_\pm)$.
	\end{theorem}
	\begin{theorem} [Bridgeman--Canary--Labourie--Sambarino\cite{BCLS}]\label{thm:BCLS}
		Suppose $U\subset\sHom(\Gamma,\sH,\sP_\pm)$ is an open ball and suppose $\xi^\pm_\rho$ are the limit maps of $\rho\in U$. Then there exists $U_0\subset U$ such that for all $\rho\in U_0$ the limit maps $\xi^\pm_\rho\in\cC^\alpha(\bdry,\sH/\sP_\pm)$ for some $\alpha>0$ and the following map is analytic:
		\begin{align*}
			\xi^\pm:U_0&\rightarrow \cC^\alpha(\bdry,\sH/\sP_\pm)\\
			\rho&\mapsto\xi^\pm_\rho.
		\end{align*}
		
	\end{theorem}
	\begin{theorem}[Guichard--Wienhard\cite{GW2}]\label{thm:GW2}
		Suppose $\sP_\pm$ and $\sQ_\pm$ respectively are two pairs of opposite virtually parabolic subgroups of $\sH$ such that $\sP_\pm$ respectively are subgroups of $\sQ_\pm$. Then $\sHom(\Gamma,\sH,\sP_\pm)\subset\sHom(\Gamma,\sH,\sQ_\pm)$.
	\end{theorem}
	\begin{remark}\label{rem:nonparano}
		Theorems \ref{thm:LGW2} and \ref{thm:BCLS} as stated above is not proved in \cite{GW2} and \cite{BCLS} but these results also hold true if we replace parabolic subgroups with virtually parabolic subgroups. This is because the proofs given in \cite{BCLS} (Theorems 6.1, 6.5, 6.6 and Lemma 6.7) only depend on the fact that the space $\sH/\sP$ is an analytic manifold and the limit map at the origin satisfies the contraction property (for more details please see Theorem 3.8 of \cite{HPS} and Theorem 5.18 of \cite{Shub}). Finally, Theorem \ref{thm:GW2} would also hold true in this setting as we would get the new limit map for free from the old limit map by composing it with the natural projection of $\sH/\sP$ onto $\sH/\sQ$ and the contraction property would still hold due to its independence from the particular collection of Euclidean norms chosen. Hence, in the remainder of this article we would use the notion of an Anosov representation to include groups of this general nature too.
	\end{remark}
	
	Recently, Stecker--Treib \cite{ST} have given a more algebraic way of characterizing Anosov representations with respect to stabilizers of oriented flags. They use techniques from \cite{KLP,KLP1,KLP2} and \cite{GGKW} to construct domains of discontinuity for oriented flag manifolds.

	\subsection{Pseudo-orthogonal groups}
	
	In this subsection we give a brief description of certain well known properties of pseudo-orthogonal groups which will be helpful later on.
	
	Let $\R^k$ be the space of all column vectors of size $k$ (i.e. matrices of size $k\times 1$) and let  $I_k$ be the identity matrix of size $k\times k$. We endow $\R^{p+q}$ with the following quadratic form of signature $(p,q)$:
	\[I_{p,q}\defeq
	\begin{bmatrix}
		I_{p} & 0\\
		0 & -I_q
	\end{bmatrix}.\]
	We denote $\R^{p+q}$ endowed with the quadratic form $I_{p,q}$ by $\R^{p,q}$ and the group of invertible linear transformations preserving the quadratic form $I_{p,q}$ by $\mathsf{O}(p,q)$. We note that any element of $\mathsf{O}(p,q)$ either has determinant $1$ or $-1$ and denote the subgroup of $\mathsf{O}(p,q)$ whose elements have determinant $1$ by $\mathsf{SO}(p,q)$ and denote the connected component of $\mathsf{SO}(p,q)$ which contains the identity transformation by $\mathsf{SO}_0(p,q)$.
	\begin{remark}\label{rem:orient}
		Suppose $g\in\mathsf{O}(p,q)$ is such that $A_g$ is a $p\times p$ matrix and
		\[g=\begin{bmatrix}
			A_g & B_g\\
			C_g & D_g
		\end{bmatrix}.\]
		Then both $A_g$ and $D_g$ are invertible. Let $\mathsf{sgn}$ denote the sign of a non-zero real number. We note that $\mathsf{O}(p,q)$ has four connected components characterized as follows: 
		\[\mathsf{O}^a_d(p,q):=\{g\mid \mathsf{sgn}\det(A_g)=a \text{ and } \mathsf{sgn}\det(D_g)=d\},\]
		where $a,d\in\{\pm\}.$ Moreover, using a quick computation we obtain that $\det(g)=\det(A_g)\det(D_g)^{-1}$. Hence, $\mathsf{SO}_0(p,q)=\mathsf{O}^+_+(p,q)$ and $\mathsf{SO}(p,q)=\mathsf{O}^+_+(p,q)\cup\mathsf{O}^-_-(p,q)$. (For more details please see Chapter 9 of \cite{neil} and Proposition 7.3 of \cite{gallier}.)
	\end{remark}
	
	\begin{lemma}\label{lem:oriso1}
		Any element of $\mathsf{O}(k,k)$ which fixes $[I_k,I_k]^t$ lies in $\mathsf{SO}_0(k,k)$.
	\end{lemma}
	\begin{proof}
		We start by observing that any matrix $X$ of dimension $2k\times2k$ which satisfy the equation $X[I_k,I_k]^t=[I_k,I_k]^t$ is of the following form for some matrices $A$ and $C$ of dimension $k\times k$,
		\[X=\begin{bmatrix}
			A&(I-A)\\
			(I-C)&C
		\end{bmatrix}\]
		Now we characterize such elements of $\mathsf{O}(k,k)$. We observe that $X^tI_{k,k}X=I_{k,k}$ impose the following set of constraints on $A$ and $C$:
		\begin{align*}
			A^tA-(I-C^t)(I-C)&=I,\\
			(I-A)^t(I-A)-C^tC&=-I,\\
			A^t(I-A)-(I-C^t)C&=0.
		\end{align*}
		The first and third equations give us
		$(A+C)=2I$ and replacing $(I-C)$ with $(A-I)$ in the first equation we deduce that
		$A^t+A=2I$. 
		
		Let us denote $(A^t-A)$ by $2M$. We note that $M$ is skew-symmetric. Hence, we obtain that $A=(I-M)$ and $C=(I+M)$.
		As $M$ is skew-symmetric, using the spectral theorem for skew-symmetric matrices we deduce that $\det(I-M)>0$ and $\det(I+M)>0$. Finally, using Remark \ref{rem:orient} we conclude our result. 
	\end{proof}
	
	\begin{lemma}\label{lem:oriso2}
		Suppose $k=\min\{p,q\}$ and $p\neq q$. Then any element of $\mathsf{SO}(p,q)$ which fixes $[I_k,0,I_k]^t$ lies inside $\mathsf{SO}_0(p,q)$. 
	\end{lemma}
	\begin{proof}
		Let us denote $(p-q)$ by $m$. We start by observing that any matrix $X$ of dimension $(p+q)\times(p+q)$ which satisfy the equation $X[I_k,0,I_k]^t=[I_k,0,I_k]^t$ is of the following form for some matrices $A,C$ of dimension $k\times k$, $B$ of dimension $|m|\times |m|$ and column vectors $u,v,w\in\R^k$,
		\[X=\begin{bmatrix}
			A&v&(I-A)\\
			u^t&B&-u^t\\
			(I-C)&w&C
		\end{bmatrix}\]
		Now we characterize all such elements of $\mathsf{O}(p,q)$. We observe that $X^tI_{p,q}X=I_{p,q}$ impose the following set of constraints on $A,B,C,u,v$ and $w$:
		\begin{align*}
			A^tA+\mathsf{sgn}(m)uu^t-(I-C^t)(I-C)&=I,\\
			(I-A)^t(I-A)+\mathsf{sgn}(m)uu^t-C^tC&=-I,\\
			A^t(I-A)-\mathsf{sgn}(m)uu^t-(I-C^t)C&=0,\\
			v^tv+\mathsf{sgn}(m)B^tB-w^tw&=\mathsf{sgn}(m)I,\\
			A^tv+\mathsf{sgn}(m)uB-(I-C^t)w&=0,\\
			(I-A^t)v-\mathsf{sgn}(m)uB-C^tw&=0.
		\end{align*}
		The last two equations give us $v=w$ and plugging it in the fourth equation we obtain that $B\in\mathsf{O}(|m|)$. Furthermore, from the first and third equations we obtain that
		$(A+C)=2I$ and replacing $(I-C)$ with $(A-I)$ in the first equation we deduce that
		$\mathsf{sgn}(m)uu^t=2I-A^t-A$. As $v=w$, replacing $(I-C)$ with $(A-I)$ in the fifth equation gives us $\mathsf{sgn}(m)uB+w=0$. It follows that, any $X\in\mathsf{O}(p,q)$ which fix $[I_k,0,I_k]^t$ are of the following form
		\[X=\begin{bmatrix}
			A&-\mathsf{sgn}(m)u&(I-A)\\
			u^t&I&-u^t\\
			(A-I)&-\mathsf{sgn}(m)u&(2I-A)
		\end{bmatrix}\begin{bmatrix}
			I&0&0\\
			0&B&0\\
			0&0&I
		\end{bmatrix}.\]
		We observe that
		\[\begin{bmatrix}
			I&0&-I\\
			0&I&0\\
			0&0&I
		\end{bmatrix}
		\begin{bmatrix}
			A&-\mathsf{sgn}(m)u&(I-A)\\
			u^t&I&-u^t\\
			(A-I)&-\mathsf{sgn}(m)u&(2I-A)
		\end{bmatrix}
		\begin{bmatrix}
			I&0&I\\
			0&I&0\\
			0&0&I
		\end{bmatrix}
		=
		\begin{bmatrix}
			I&0&0\\
			*&I&0\\
			*&*&I
		\end{bmatrix}.\]
		Hence, for $X\in\mathsf{SO}(p,q)$ we have $\det(B)=1$.
		
		Let us denote $(A^t-A)$ by $M$. We note that $M$ is skew-symmetric. Hence, for $p>q$ we obtain that $\det(2I-A)=\det(I+(M+uu^t)/2)$ and for $p<q$ we obtain that
		$\det(A)=\det(I-(M-uu^t)/2)$. As $M$ is skew-symmetric we have
		$\det(I+(M+uu^t)/2)=\det(I-(M-uu^t)/2).$	Moreover, as $I+uu^t/2$ is a symmetric positive definite matrix, using Cholesky factorization (see Corollary 7.2.9 of \cite{HornJohn}) we obtain the existence of a real lower triangular matrix $T$ with positive entries on the diagonal such that $I+uu^t/2=TT^t$. Hence,
		\[\det\left(I+(uu^t+M)/2\right)=\det(T)\det(I+T^{-1}(M/2)(T^t)^{-1})\det(T^t).\]
		As $T^{-1}M(T^t)^{-1}$ is skew-symmetric, using the spectral theorem for skew-symmetric matrices we deduce that $\det(I+T^{-1}(M/2)(T^t)^{-1})>0$. Finally, using Remark \ref{rem:orient} we conclude our result.
	\end{proof}

	\subsection{Maximal isotropic spaces}
	
	Let $V$ be a subspace of $\R^{p,q}$. Then $V$ is called \emph{isotropic} if and only if for all $v\in V$ we have $v^tI_{p,q}v=0$. In this subsection we list and demonstrate a few properties of maximal isotropic subspaces. These properties are used many times in the later part of this text. Although we expect these results to be well known, we were unable to find appropriate sources in the literature for them. Keeping in mind their centrality we have decided to include them.
	
	\begin{remark}
		Let $V$ be a maximal isotropic space of $\R^{p,q}$. Then there exists a maximal isotropic space $W$ which is transverse to $V$ i.e. $V^\perp\cap W=\{0\}$. Note that $V$, $W$ are spaces of dimension $k=\min\{p,q\}$ and the existence of $W$ is guaranteed by the fact that $I_{p,q}V$ is a maximal isotropic space which is transverse to $V$.
	\end{remark}
	
	\begin{notation}
		We denote the space spanned by the column vectors of a matrix $X$ by $\mathsf{cspan}(X)$.   
	\end{notation}
	
	\begin{lemma}\label{lem:stiso}
		Suppose $(V,W)$ is a pair of transverse maximal isotropic subspaces of $\R^{p,q}$, suppose $k=\min\{p,q\}$ and $J=I_{2k-1,1}$. Then 
		\begin{enumerate}
			\item for $p=q=k$, there exists $g\in\mathsf{SO}_0(p,q)$ such that either of the following holds:
			\begin{align*}
				(V,W)&=(g \mathsf{cspan}([I_k,I_k]^t),g \mathsf{cspan}([I_k,-I_k]^t)),\\
				(V,W)&=(gJ\mathsf{cspan}([I_k,I_k]^t),gJ\mathsf{cspan}([I_k,-I_k]^t)),
			\end{align*}
			\item and for $p\neq q$, there exists $g\in\mathsf{SO}_0(p,q)$ such that 
			\[(V,W)=(g\mathsf{cspan}([I_k,0,I_k]^t),g\mathsf{cspan}([I_k,0,-I_k]^t)).\]
		\end{enumerate}
	\end{lemma}
	\begin{proof}
		Let $B=\{ v_1,\dots,v_k\}$ be a basis of $V$. Let $B_j=B\setminus\{v_j\}$. Then $\dim(B_j^\perp\cap W)\geqslant1$. Choose $w_j\in (B_j^\perp\cap W)$ such that $v_j^tI_{p,q}w_j=1/2$. We note that the existence of such a $w_j$ is guaranteed by the maximality of $V$. 
		
		$\diamond$ Case $p=q=k$: We observe that $V^\perp\cap W^\perp=\{0\}$ and deduce that \[h=[v_1+w_1,\dots,v_k+w_k,v_1-w_1,\dots,v_k-w_k]\in\mathsf{O}(k,k).\]
		Hence, $h[I_k,I_k]^t=2[v_1,\dots,v_k]$ and $h[I_k,-I_k]^t=2[w_1,\dots,w_k]$. Now we use Remark \ref{rem:orient} and observe that upto a sign change in $v_1$ (consequently in $w_1$) we can make sure that $h\in\mathsf{O}^+_+(k,k)\cup\mathsf{O}^+_-(k,k)$. Now, if $h\in\mathsf{SO}_0(k,k)$, then we choose $g=h$. Otherwise, we choose $g=hJ$. We use Remark \ref{rem:orient} to observe that $g\in\mathsf{SO}_0(k,k)$, $gJ[I_k,I_k]^t=2[v_1,\dots,v_k]$ and $gJ[I_k,-I_k]^t=2[w_1,\dots,w_k]$.
		
		$\diamond$ Case $p\neq q$: As both $V$ and $W$ are maximal, any non-zero $u\in V^\perp\cap W^\perp$ satisfy $u^tI_{p,q}u\neq0$. As $u^tI_{p,q}u=(-u)^tI_{p,q}(-u)$, using continuity we obtain that either $I_{p,q}$ restricts to a positive definite form on $V^\perp\cap W^\perp$ or a negative definite one. In either case, using a Gram-Schmidt process we obtain an orthonormal basis $u_1,\dots,u_{|p-q|}$. Now we use Remark \ref{rem:orient} and observe that upto a sign change in $v_1$ (consequently in $w_1$ too) and in $u_1$,
		\[g=[v_1+w_1,\dots,v_k+w_k,u_1,\dots,u_{|p-q|},v_1-w_1,\dots,v_k-w_k]\in\mathsf{SO}_0(p,q).\]
		Hence, $V$ is spanned by the column vectors of $g[I_k,0,I_k]$ and $W$ is spanned by the column vectors of $g[I_k,0,-I_k]$.
	\end{proof}
	\begin{notation}
		Suppose $K\in\mathsf{GL}(k,\R)$. We consider the following functions:
		\[f_{k,0}(K):=
		\begin{bmatrix}
			\frac{K^{-1}+K^t}{2}&\frac{K^{-1}-K^t}{2}\\
			\frac{K^{-1}-K^t}{2}&\frac{K^{-1}+K^t}{2}
		\end{bmatrix}
		\text{ and }
		f_{k,m}(K):=\begin{bmatrix}
			\frac{K^{-1}+K^t}{2}&0&\frac{K^{-1}-K^t}{2}\\
			0&I&0\\
			\frac{K^{-1}-K^t}{2}&0&\frac{K^{-1}+K^t}{2}
		\end{bmatrix}\]
		and observe that $f_{k,m}(K)\in\mathsf{SO}(k+m,k)\cap\mathsf{SO}(k,m+k)$ for all $m\geqslant0$.	
	\end{notation}
	\begin{lemma}
		Suppose $V, V^\prime\subset\R^{k,k}$ are such that $V$ is the span of the column vectors of $[I_k,I_k]^t$ and $V^\prime$ is the span of the column vectors of $[I_k,I_{k-m,m}]^t$ with odd $m$. Then there does not exist any $g\in\mathsf{SO}(k,k)$ such that $gV=V^\prime$.
	\end{lemma}
	\begin{proof}
		We observe that $J=I_{2k-m,m}\in\mathsf{O}(k,k)$. We will prove this result by contradiction. If possible let us assume that $gV=V^\prime$ for some $g\in\mathsf{SO}(k,k)$ then $JgV=V$. Hence, $Jg[I_k,I_k]^t=[I_k,I_k]^tK$ for some invertible $k\times k$ matrix $K$. We observe that $f_{k,0}(K)Jg[I_k,I_k]^t=[I_k,I_k]^t$. Hence, using Lemma \ref{lem:oriso1} we have $f_{k,0}(K)Jg\in\mathsf{SO}_0(k,k)$. As both $f_{k,0}(K)$ and $g$ have determinant $1$, we deduce that $\det(J)=1$, a contradiction.
	\end{proof}
	
	\begin{lemma}\label{lem:maxorient}
		Any element of $\mathsf{SO}_0(p,q)$ which preserves a maximal isotropic space also preserves the orientation on it.
	\end{lemma}
	\begin{proof}
		Let $V$ be a maximal isotropic subspace of $\R^{p,q}$ and let $g\in\mathsf{SO}_0(p,q)$ be such that $gV=V$. 
		
		$\diamond$ Case $p=q=k$: We use Lemma \ref{lem:stiso} to obtain that there exists $h\in\mathsf{O}(k,k)$ such that 
		$h^{-1}gh[I_k,I_k]^t=[I_k,I_k]^tK$ for some $k\times k$ matrix $K$. We observe that $h^{-1}gh\in\mathsf{SO}_0(k,k)$. Hence, $f_{k,0}(K)h^{-1}gh\in\mathsf{SO}(k,k)$. Also, $f_{k,0}(K)h^{-1}gh[I_k,I_k]^t=[I_k,I_k]^t$. Now we use Lemma \ref{lem:oriso1} and obtain that $f_{k,0}(K)h^{-1}gh\in\mathsf{SO}_0(k,k)$. Therefore, $f_{k,0}(K)\in \mathsf{SO}_0(k,k)$ and by Remark \ref{rem:orient} it follows that $\det(K)>0$.
		
		$\diamond$ Case $p\neq q$: We denote $|p-q|$ by $m$ and $\min\{p,q\}$ by $k$. We use Lemma \ref{lem:stiso} to obtain that there exists $h\in\mathsf{O}(p,q)$ such that 
		$h^{-1}gh[I_k,0,I_k]^t=[I_k,0,I_k]^tK$ for some $k\times k$ matrix $K$. We observe that $h^{-1}gh\in\mathsf{SO}_0(p,q)$. Hence, $f_{k,m}(K)h^{-1}gh\in\mathsf{SO}(p,q)$. Also, \[f_{k,m}(K)h^{-1}gh[I_k,0,I_k]^t=[I_k,0,I_k]^t.\] 
		Now we use Lemma \ref{lem:oriso2} and obtain that $f_{k,m}(K)h^{-1}gh\in\mathsf{SO}_0(p,q)$. Therefore, $f_{k,m}(K)\in \mathsf{SO}_0(p,q)$ and by Remark \ref{rem:orient} it follows that $\det(K)>0$.
	\end{proof}
	
	\begin{lemma}\label{lem:choice}
		Let $W$ be a $(k-1)$-dimensional isotropic subspace of $\R^{k,k}$. Then $W^\perp$ contain exactly two different maximal isotropic subspaces. One of them lies in the orbit of $\mathsf{cspan}([I_k,I_k]^t)$ under the action of $\mathsf{SO}_0(n,n)$ and the other one lies in the orbit of $\mathsf{cspan}(J[I_k,I_k]^t)$ under the action of $\mathsf{SO}_0(n,n)$.
	\end{lemma}
	\begin{proof}
		As the dimension of $W$ is $(k-1)$, we can extend $W$ to a maximal isotropic subspace $V$ of $\R^{k,k}$. Clearly, $V=V^\perp\subset W^\perp$. Moreover, let $U$ be a maximal isotropic subspace of $\R^{k,k}$ transverse to $V$. We choose a basis $B=\{v_1,\dots,v_k\}$ of $V$ such that $W=\mathsf{span}(B\setminus\{v_k\})$. Moreover, suppose $\{u_1,\dots,u_k\}$ is a basis of $U$ such that $u_j\in(B\setminus\{v_j\})^\perp\cap U$ and $u_j^tI_{k,k}v_j=1$.
		Note that by construction $u_k\in W^\perp$ and $u_k^tI_{k,k}v_k=1$. Hence, $W\oplus\R u_k$ is a maximal isotropic subspace of $W^\perp$ different from $V$.
		
		Also, $W^\perp=W\oplus \R v_k\oplus \R u_k$. Hence, for any $w\in W^\perp$, there exist $a,b\in\R$ such that $w-av_k-bu_k\in W$. Moreover, if $\langle w \mid w \rangle=0$, then $ab=0$ and it follows that either $w\in V$ or $w\in W\oplus\R u_k$.
		
		Finally, for $J=I_{2k-1,1}$ we observe that
		\[h=[v_1+u_1,\dots,v_k+u_k,v_1-u_1,\dots,v_k-u_k]\in\mathsf{O}(k,k),\]
		$h[I_k,I_k]^t=2[v_1,\dots,v_{k-1},v_k]$ and $hJ[I_k,I_k]^t=2[v_1,\dots,v_{k-1},u_k]$. Moreover, up to a sign change we can choose $v_1$ (consequently $u_1$), such that $h\in\mathsf{O}^+_+(k,k)\cup\mathsf{O}^+_-(k,k)$. Finally, we conclude by noting that $h\in\mathsf{O}^+_\pm(k,k)$ if and only if $hJ\in\mathsf{O}^+_\mp(k,k)$.
	\end{proof}
	
	\section{Proper actions on $\bH$}\label{sec:specialAnosov}
	In this section we demonstrate a necessary and sufficient criterion for proper actions of a hyperbolic group $\Gamma$ on $\bH$ in terms of certain special Anosov representations. This statement is inspired from Danciger--Zhang \cite{DZ} and we provide an alternate proof. The constructions, statements and proofs in this section parallels similar constructions, statements and proofs given for affine spaces in \cite{GT}.
	
	\subsection{Proto-neutral sections}\label{sec:proto-neutral}
	In this subsection we introduce a few notions in order to state the necessary and sufficient criterion of Theorem \ref{thm1}. 
	
	We consider $\R^{2n}$ and henceforth for all vectors $v,w$ we denote $v^tI_{n,n} w$ by $\langle v\mid w\rangle$. Let $e_j\in\R^{2n}$ be the column vector whose only non-vanishing entry is the $j$-th entry and its $j$-th entry is $1$. We consider the embedding of $\R^{2n-1}$ inside $\R^{n,n}$ which is spanned by the vectors $e_1,\dots,e_{2n-1}$. We note that the quadratic form $I_{n,n}$ induces a form of signature $(n,n-1)$ on $\R^{2n-1}$. Henceforth, we denote this embedding along with the induced form by $\bR$, the group $\mathsf{SO}_0(n,n)$ by $\G$ and the subgroup of $\G$ which fixes $e_{2n}$ by $\G_0$ i.e.
	\[\G_0:=\{g\in\G\mid ge_{2n}=e_{2n}\}.\]
	We observe that $ge_{2n}=e_{2n}$ imply $g^te_{2n}=e_{2n}$. Hence, $\Go$ is a subgroup of $\mathsf{O}(n,n-1)$. Moreover, we use Remark \ref{rem:orient} to obtain that $\Go\cong\mathsf{SO}_0(n,n-1)$. We fix the following four subspaces:
	\begin{align*}
		V_\pm&:=\mathsf{span}\{e_j\pm e_{n+j}\mid j=1,\dots,n\},\\
		W_\pm&:=\mathsf{span}\{e_j\pm e_{n+j}\mid j=1,\dots,n-1\}.
	\end{align*}
	We observe that $V_\pm$ are a transverse pair of maximal isotropic subspaces of $\R^{n,n}$ and $W_\pm$ are a transverse pair of maximal isotropic subspaces of $\R^{n,n-1}$. We define $\oV_\pm$ to be the space $V_\pm$ along with the orientation coming from the ordered basis $(e_1\pm e_{n+1},\dots,e_n\pm e_{2n})$ and $\oW_\pm$ to be the space $W_\pm$ along with the orientation coming from $(e_1\pm e_{n+1},\dots,e_{n-1}\pm e_{2n-1})$. Furthermore, we define
	\[\sP^\pm\defeq\s_{\G}(\oW_\pm) \text{ and }\sP_0^\pm\defeq\s_{\Go}(\oW_\pm).\]
	\begin{remark}
		We use Lemma \ref{lem:maxorient} to deduce that $\sP_0^\pm=\s_{\Go}(W_\pm)$.
	\end{remark}
	We observe that $\sP_0^\pm=\sP^\pm\cap\Go$ and $\Go/\sP_0^\pm\subset\G/\sP^\pm$. Let $\sL_0:=\sP_0^+\cap\sP_0^-$ and $\sL:=\sP^+\cap\sP^-$. 
	\begin{remark}
		We observe that $W_+\oplus W_-\oplus\R e_n=\R^{2n-1}$. As elements of $\sL_0$ preserve orientations on $W_\pm$, we obtain that they preserve orientations on $\R e_n$ too. Also, for any $g\in\Go$ we have $\langle ge_n\mid ge_n\rangle=1$. It follows that the elements of $\sL_0$ preserve the vector $e_n$.
	\end{remark}
	We use the above Lemma and obtain a well defined map:
	\begin{align*}
		\nu:\Go/\sL_0 &\longrightarrow \R^{2n+1}\\
		[g]&\longmapsto ge_n.
	\end{align*}
	The map $\nu$ is called the \textit{neutral section}.
	\begin{remark}\label{rem:nuor}
		Any element in $\Go/\sL_0$ can be interpreted as a tuple of transverse null subspaces. The neutral section keeps track of the directional intersection of a pair of transverse null subspaces (for more details about its usage please see Remark \ref{rem:nu}). In particular, for $n=2$, let $(V,W)$ be a pair of transverse null subspaces of $\R^3$. Then $V,W$ are tangent planes to the light-cone, they touch the light-cone at two distinct lines. In fact, $V\cap W$ is also a line. Let $v\in V$, $w\in W$ be such that $v,w$ lie in the upper light-cone and let $u\in V\cap W$ be such that $[v,u,w]$ give the standard orientation on $\R^3$. Then we have  
		\[\langle u\mid u\rangle>0 \text{ and }\nu(V,W)=\frac{u}{\sqrt{\langle u\mid u\rangle}}.\]
	\end{remark}
	
	Henceforth, we fix an Euclidean norm $\|\cdot\|$ on $\R^{2n}$ and choose the following two vectors 
	\[v_+:=\frac{e_n+e_{2n}}{2}\text{ and }v_-:=\frac{e_n-e_{2n}}{2}.\]
	\begin{lemma}
		Suppose $\sL=\sP^+\cap\sP^-$. Then the elements of $\sL$ preserve the orientations on the isotropic lines $\R v_\pm$.
	\end{lemma}
	\begin{proof}
		We observe that $V_\pm=W_\pm\oplus\R v_\pm$. As elements of $\sL$ preserve the oriented space $\oW_\pm$, we deduce that they preserve orientations on $\R v_+\oplus\R v_-$. Hence, the action of $\sL$ on $\R v_+\oplus\R v_-$ is isomorphic to the action of $\mathsf{SO}(1,1)$ on $\R^{1,1}$. As the action of $\mathsf{SO}(1,1)$ on $\R^{1,1}$ preserve the individual isotropic lines, we deduce that the elements of $\sL$ preserve the lines $\R v_\pm$. Hence, the action of $\sL$ preserves $V_\pm$. Moreover, $V_\pm$ are maximal isotropic subspaces of $\R^{n,n}$. Hence, by Lemma \ref{lem:maxorient} we obtain that the action of $\sL$ preserves $\oV_\pm$. In fact, the action of $\sL$ also preserves $\oW_\pm$. It follows that the elements of $\sL$ preserves the orientations on the isotropic lines $\R v_\pm$.
	\end{proof}
	We use the above Lemma and obtain two well defined maps:
	\begin{align*}
		\nu_\pm: \G/\sL&\longrightarrow\R^{2n}\\
		[g] &\longmapsto {gv_\pm}/{\|gv_\pm\|}.
	\end{align*}
	We also observe that for all $g,h\in\G$, $\|\nu_\pm([g])\|=1$ and
	\[h\nu_\pm([g])=\frac{\|hgv_\pm\|}{\|gv_\pm\|}\nu_\pm([hg])=\left\|h\nu_\pm([g])\right\|\nu_\pm(h[g]).\]
	We call the maps $\nu^\pm$ as \textit{proto-neutral sections} as they play a role similar to the one played by the neutral section $\nu$ in \cite{GT}.
	
	\begin{remark}
		Any element in $\G/\sL$ can be interpreted as a tuple of oriented $(n-1)$-dimensional isotropic subspaces of $\R^{n,n}$ whose orthogonal spaces are transverse to each other. The intersection of these two orthogonal spaces is a space of dimension two which contains exactly two isotropic lines. The proto-neutral sections provide a consistent way of distinguishing one isotropic line from the other of the two isotropic lines (for more details about its usage please see Remark \ref{rem:protonu}).
	\end{remark}
	
	\begin{lemma}\label{lem:protoneu}
		Suppose $x,y\in\{\pm\}$ and  $g\in\G$ is such that $g\in\sP^x$ then 
		\[W_x+\R\nu_y([g])=W_x+\R v_y.\]
	\end{lemma}
	\begin{proof}
		Suppose $g\in\sP^x$. Hence, $gW_x=W_x$ and $(W_x+\R\nu_y([g]))$ is in the $\G$-orbit of $(W_x+\R v_y)$. As $W_x\subset (W_x+\R v_y)$ and $(W_x+\R v_y)$ is maximally isotropic, we have $(W_x+\R v_y)\subset W_x^\perp$. It follows that 
		\[g(W_x+\R v_y)\subset gW_x^\perp=(gW_x)^\perp=W_x^\perp.\] 
		As $g(W_x+\R v_y)$ is in the $\G$-orbit of $(W_x+\R v_y)$ and both of them are maximal isotropic subspaces of $W_x^\perp$, we use Lemma \ref{lem:choice} and conclude that $g(W_x+\R v_y)=(W_x+\R v_y)$ and our result follows.
	\end{proof}
	\begin{lemma}\label{lem:nu0}
		Suppose $x,y\in\{\pm\}$ and $g,h\in\G$ satisfy $g^{-1}h\in\sP^x$. Then
		\begin{align*}
			\langle\nu_y([g])\mid\nu_y([h])\rangle&=0,\\
			\langle\nu_y([g])\mid\nu_{-y}([h])\rangle&\neq0.
		\end{align*}
	\end{lemma}
	\begin{proof}
		As $g^{-1}h\in\sP^x$, using Lemma \ref{lem:protoneu} we deduce that both $g^{-1}hv_y$ and $v_y$ lie in the same maximal isotropic subspace $W_x\oplus\R v_y$. Hence, $\langle v_y\mid g^{-1}hv_y\rangle=0$. Moreover, $g^{-1}hv_y\in W_x\oplus\R v_y$ and $g^{-1}hv_y\notin W_x$. Hence, $g^{-1}hv_y+av_y\in W_x$ for some $a\neq0$. It follows that $\langle v_{-y}\mid g^{-1}hv_y\rangle=a\langle v_{-y}\mid v_y\rangle\neq0$.
		Therefore,
		\begin{align*}
			\langle\nu_y([g])\mid\nu_y([h])\rangle&=\frac{\langle gv_y\mid hv_y\rangle}{\|gv_y\|\|hv_y\|}=\frac{\langle v_y\mid g^{-1}hv_y\rangle}{\|gv_y\|\|hv_y\|}=0,\\
			\langle\nu_y([g])\mid\nu_{-y}([h])\rangle&=\frac{\langle gv_y\mid hv_{-y}\rangle}{\|gv_y\|\|hv_{-y}\|}=\frac{\langle v_y\mid g^{-1}hv_{-y}\rangle}{\|gv_y\|\|hv_{-y}\|}\neq0,
		\end{align*}
		and our result follows.
	\end{proof}
	
	\begin{remark}\label{rem:signu}
		We note that $I_{n,n}\in\mathsf{O}^+_+(n,n)$ for even $n$ and $I_{n,n}\in\mathsf{O}^+_-(n,n)$ for odd $n$. Hence, for even $n$ we deduce that $\nu_\pm([I_{n,n}])=\nu_\mp([I])$, 
		\[\nu_\pm([gI_{n,n}])=\frac{g\nu_\pm([I_{n,n}])}{\left\|g\nu_\pm([I_{n,n}])\right\|}=\frac{g\nu_\mp([I])}{\left\|g\nu_\mp([I])\right\|}=\nu_\mp([g]).\]	
		Moreover, $-I\in\mathsf{O}^+_+(n,n)$ for even $n$ and we have $\nu_\pm([-g])=-\nu_\pm([g])$.
	\end{remark}

	\subsection{Proto-Labourie-Margulis invariants} \label{sec:pLMinv}
	In this subsection we define the notion of diffused eigenvalues and call them proto-Labourie-Margulis invariants due to their relation with diffused Margulis invariants introduced by Labourie. 
	
	Suppose $\Gamma$ is a word hyperbolic group and $\flow$ is its Gromov flow space. We denote the period of the periodic orbit in $\flow$ corresponding to an infinite order element $\gamma\in\Gamma$ by $l(\gamma)$ and the flow invariant probability measure on this periodic orbit by $\mu_\gamma$. Now for any  $\rho\in\sHom(\Gamma,\G,\sP^\pm)$, any infinite order element $\gamma\in\Gamma$ and $t\in\R$ we have,  
	\[\nu_\pm(\xi_\rho(\gamma^+,\gamma^-,t))=\nu_\pm(\xi_\rho(\gamma^+,\gamma^-,0)).\]
	Moreover, we denote $(\gamma^+,\gamma^-,0)$ by $p_\gamma$ and denote $\nu_\pm\circ\xi_\rho$ by $\nu^\pm_\rho$. We observe that $\nu^\pm_\rho(p_\gamma)$ is an eigenvector of $\rho(\gamma)$. Indeed, we have
	\[\nu^\pm_\rho(p_\gamma)=\nu_\pm(\rho(\gamma)\xi_\rho(p_\gamma))=\frac{\rho(\gamma)\nu^\pm_\rho(p_\gamma)}{\|\rho(\gamma)\nu^\pm_\rho(p_\gamma)\|}.\]
	As $\langle\rho(\gamma) \nu^+_\rho(p_\gamma)\mid \rho(\gamma)\nu^-_\rho(p_\gamma)\rangle=\langle \nu^+_\rho(p_\gamma)\mid \nu^-_\rho(p_\gamma)\rangle$ it follows that 
	\[\|\rho(\gamma) \nu^+_\rho(p_\gamma)\|\|\rho(\gamma)\nu^-_\rho(p_\gamma)\|=1.\]
	\begin{notation}\label{not:lam}
		We denote the eigenvalues of $\rho(\gamma)$ corresponding to the eigenvectors $\nu^\pm_\rho(p_\gamma)$ by $\lambda^\pm_\rho(\gamma)$. We observe that
		\[\lambda^\pm_\rho(\gamma)=\|\rho(\gamma)\nu^\pm_\rho(p_\gamma)\| \text{ and }\lambda^+_\rho(\gamma)\lambda^-_\rho(\gamma)=1.\]
	\end{notation}
	
	Suppose $\{\rho_t\}_{t\in(-1,1)}\subset\sHom(\Gamma,\G,\sP^\pm)$ is an analytic one parameter family. We now present a partition of unity type argument to construct a family of Euclidean norms indexed by points in the Gromov flow space $\flow$ and $\sP^\pm$-Anosov representations in $\G$. Recall that $\pi:\cflow\rightarrow\flow$ is the standard projection map. 
	\begin{remark}\label{rem:norm}
		As $\flow$ is compact, there exist $\{V_i\}_{i=1}^k$ such that $V_i\subset\cflow$ are small open balls and $\cup_{i=1}^k\pi(V_i)=\flow$. Hence $\cup_{\gamma\in\Gamma}\cup_{i=1}^k\gamma V_i=\cflow$. We know from Section 8.2 of \cite{GT} that there exist maps 
		\[\{f_i:\flow\rightarrow\R^+\}_{i=1}^k\] 
		with $\mathsf{Supp}(f_i)\subset \pi(V_i)$ such that the functions $f_i$ are H\"older continuous and differentiable along flow lines with $\sum_{i=1}^kf_i=1$. We use this to construct a collection 
		\[\left\{\|\cdot\|_p^t\mid(\rho_t,p)\in\sHom(\Gamma,\G,\sP^\pm)\times\cflow\right\}\] 
		of Euclidean norms on $\R^{2n}$ indexed by $\sHom(\Gamma,\G)\times\cflow$ such that:
		\begin{enumerate}
			\item it is H\"older continuous in the variable $p\in\cflow$,
			\item it is smooth along the flow lines of $\{\phi_s\}_{s\in\R}$,
			\item it is analytic along the variable $\rho_t$,
			\item it is equivariant i.e. $\|\rho_t(\gamma)v\|_{\gamma p}^t=\|v\|_p^t$ for all $v\in\R^{2n}$ and $\gamma\in\Gamma$.
		\end{enumerate}
		We start by considering the fixed Euclidean norm $\|\cdot\|$ on $\R^{2n}$ mentioned in Section \ref{sec:proto-neutral}. We observe that for any $p\in\Gamma V_i$ there exists a unique $\gamma_{p,i}$ such that $\gamma_{p,i}p\in V_i$. Note that in such a situation $\gamma_{\eta p,i}\eta=\gamma_{p,i}$. We define
		\[\|v\|^t_{p,i}\defeq\|\rho_t(\gamma_{p,i})v\|\]
		for all $v\in\R^{2n}$ and for any $p\in\cflow$ we define:
		\[\|v\|^t_p\defeq\Sigma_{i=1}^kf_i(\pi(p))\|v\|^t_{p,i}.\]
		We check that this collection of norms are equivariant. Indeed, as
		\begin{align*}
			\|\rho_t(\gamma)v\|^t_{\gamma p}&=\Sigma_{i=1}^kf_i(\pi(\gamma p))\|\rho_t(\gamma)v\|^t_{\gamma p,i}=\Sigma_{i=1}^kf_i(\pi(p))\|\rho_t(\gamma_{\gamma p,i})\rho_t(\gamma)v\|\\
			&=\Sigma_{i=1}^kf_i(\pi(p))\|\rho_t(\gamma_{\gamma p,i}\gamma)v\|=\Sigma_{i=1}^kf_i(\pi(p))\|\rho_t(\gamma_{p,i})v\|\\
			&=\Sigma_{i=1}^kf_i(\pi(p))\|v\|^t_{p,i}=\|v\|^t_p.
		\end{align*}
		Moreover, it follows from our construction that this collection of norms satisfies all the first three conditions listed above. More details about properties of these kinds of constructions can be found in \cite{HPS} and \cite{Shub} (see also \cite{Ghosh1,Ghosh2} and \cite{BCLS}).
	\end{remark}
	
	\begin{notation}
		Henceforth, to simplify our notations we denote $\nu_\pm\circ\xi$ (resp. $\nu_\pm\circ\xi_t$) by $\nu^\pm$ (resp. $\nu^\pm_t$) and observe that by definition it is invariant under the flow $\phi$.
	\end{notation}
	
	Suppose $\xi:\cflow\to(\G/\sL)$ are the limit maps corresponding to the representations $\rho\in\sHom(\Gamma,\G,\sP^\pm)$. We recall the proto-neutral sections from Section \ref{sec:proto-neutral} and using the above collection of norms we define:
	\begin{align*}
		\sigma^\pm:\cflow&\longrightarrow\R^{2n}\\
		p&\longmapsto \nu^\pm(p)/\|\nu^\pm(p)\|_p.
	\end{align*}
	We call these maps \emph{proto-neutralised sections}. We observe that for all $p\in\cflow$, $\langle \sigma^+(p)\mid\sigma^-(p)\rangle>0$, and for $\gamma\in\Gamma$ we obtain that,
	\[\sigma^\pm(\gamma p)
	=\frac{\nu^\pm(\gamma p)}{\|\nu^\pm(\gamma p)\|_{\gamma p}}
	=\frac{\rho(\gamma)\nu^\pm(p)}{\|\rho(\gamma)\nu^\pm(p)\|_{\gamma p}}=\frac{\rho(\gamma)\nu^\pm(p)}{\|\nu^\pm(p)\|_p}=\rho(\gamma)\sigma^\pm(p).\]
	Also, we deduce that
	\[\sigma^\pm(\phi_sp)=\frac{\nu^\pm(\phi_sp)}{\|\nu^\pm(\phi_sp)\|_{\phi_sp}}=\frac{\nu^\pm(p)}{\|\nu^\pm(p)\|_{\phi_sp}}=\frac{\|\nu^\pm(p)\|_{p}\sigma^\pm(p)}{\|\nu^\pm(p)\|_{\phi_sp}}.\]
	\begin{remark}
		We recall that for $p_\gamma=(\gamma^+,\gamma^-,0)$ we have $\phi_{l(\gamma)} p_\gamma=\gamma p_\gamma$ and hence $\|\nu^\pm(p_\gamma)\|_{\phi_{l(\gamma)} p_\gamma}=\lambda^\mp(\gamma)\|\nu^\pm(p_\gamma)\|_{p_\gamma}$.	It follows that 
		\[\sigma^\pm(\gamma p_\gamma)=\lambda^\pm(\gamma)\sigma^\pm(p_\gamma).\]
	\end{remark}
	\begin{notation} 
		We consider the proto-neutralised sections and define
		\[\nabla_\phi\sigma^\pm(p)\defeq\left.\frac{\partial}{\partial s}\right|_{s=0}\sigma^\pm(\phi_sp)=-\left.\frac{\partial}{\partial s}\right|_{s=0}\log\|\nu^\pm(p)\|_{\phi_sp}\sigma^\pm(p).\]
		Furthermore, for all $p\in\cflow$ we define:
		\[f^\pm(p)\defeq\frac{\langle \nabla_\phi\sigma^\pm(p)\mid\sigma^\mp(p)\rangle}{\langle \sigma^+(p)\mid\sigma^-(p)\rangle}=-\left.\frac{\partial}{\partial s}\right|_{s=0}\log\|\nu^\pm(p)\|_{\phi_sp}.\]
	\end{notation}
	\begin{remark} \label{rem:pLMinv}
		As the action of $\Gamma$ commutes with the flow and the proto-neutralised sections are $\Gamma$-equivariant, we obtain that the functions $f^\pm$ are $\Gamma$-invariant i.e. $f^\pm(\gamma p)=f^\pm(p)$ for all $\gamma\in\Gamma$ and $p\in\cflow$. Hence, $f^\pm$ induce functions on $\flow$. We abuse notation and denote these functions by $f^\pm$ too.
	\end{remark}
	\begin{definition}
		Let $f,g:\flow \rightarrow\R$ be two H\"older continuous functions. Then $f$ is said to be Liv\v sic cohomologous to $g$ if there exists a function $h:\flow\rightarrow\R$ which is differentiable along the flow $\phi_t$ and satisfies the following property:
		\[f-g=\left.\frac{\partial}{\partial t}\right|_{t=0}h\circ\phi_t.\]	
	\end{definition}
	\begin{definition}
		The Liv\v sic cohomology classes $[f^+_t]$ (resp. $[f^-_t]$), of the functions on $\flow$ mentioned in Remark \ref{rem:pLMinv}, are called \emph{left} (resp. \emph{right}) \emph{proto-Labourie-Margulis invariants} of $\rho_t$. 
	\end{definition}
	
	\begin{proposition}\label{prop:derlminv}
		Suppose $\mu_\gamma$ is the flow invariant probability measure supported on the periodic orbit of $\flow$ corresponding to $\gamma\in\Gamma$ with period $l(\gamma)$ and $f^\pm:\flow\rightarrow\R$ is defined as above. Then 
		\[\int f^\pm d\mu_\gamma=\frac{\log\lambda^\pm(\gamma)}{l(\gamma)}.\]
	\end{proposition}
	\begin{proof}
		Suppose $p_\gamma$ belongs to the periodic orbit corresponding to $\gamma$. We deduce that 
		\[\|\nu^\pm(p_\gamma)\|_{\phi_{l(\gamma)}p_\gamma}=\|\nu^\pm(p_\gamma)\|_{\gamma p_\gamma}=\|\rho(\gamma)^{-1}\nu^\pm(p_\gamma)\|_{p_\gamma}.\]
		Hence, we obtain that
		\[\int_0^{l(\gamma)}f^\pm(\phi_sp_\gamma)ds=\log\|\nu^\pm(p_\gamma)\|_{p_\gamma}-\log\|\nu^\pm(p_\gamma)\|_{\phi_{l(\gamma)}p_\gamma}=\log\lambda^\pm(\gamma),\]
		and our result follows.
	\end{proof}
	\begin{remark}\label{rem:invcohom}
		As $\lambda^+(\gamma)\lambda^-(\gamma)=1$, we obtain that
		\[\int \frac{1}{2}(f^++f^-)d\mu_\gamma=0 \text{ and } \int \frac{1}{2}(f^+-f^-)d\mu_\gamma=\frac{\log\lambda^+(\gamma)}{l(\gamma)}\] 
		for all appropriate $\gamma\in\Gamma$. Hence, by Liv\v sic's Theorem \cite{Liv} we obtain that the functions $(f^++f^-)/2$ are Liv\v sic cohomologous to zero and the the functions $(f^+-f^-)/2$ are Liv\v sic cohomologous respectively to $f^+$.
	\end{remark}
	
	\subsection{An equivalent criterion}
	In this subsection we prove an equivalent criterion for proper actions of hyperbolic groups $\Gamma$ on $\bH$ in terms of proto-Labourie-Margulis invariants.
	\begin{lemma}\label{lem:proto}
		Suppose $p\in\cflow$ and $\sigma^\pm$ are as in Subsection \ref{sec:pLMinv}. Then there exists $h(p)\in\G$ such that $\nu^\pm(p)=\nu_\pm([h(p)])$, $h(\phi_sp)=h(p)$ and
		\[\frac{\sigma^\pm(p)}{\sqrt{2\langle\sigma^+(p)\mid\sigma^-(p)\rangle}}=\frac{1}{2}\left(\sqrt{\frac{\|h(p)v_-\|_{p}}{\|h(p)v_+\|_{p}}}\right)^{\pm1}h(p) v_\pm.\]
		Moreover, there exists some $g(p)\in\G$ such that
		\[\frac{(\sigma^+(p)-\sigma^-(p))}{\sqrt{2\langle\sigma^+(p)\mid\sigma^-(p)\rangle}}=g(p)e_{2n}.\]
	\end{lemma}
	\begin{proof}
		We observe that $\nu^\pm(p)=\nu_\pm([h(p)])=(h(p)v_\pm)/\|h(p)v_\pm)\|$ for some $h(p)\in\G$. We can choose $h(p)$ such that it does not depend on the flow. Hence, we obtain that
		\[\sigma^\pm(p)=\frac{h(p)v_\pm}{\|h(p)v_\pm\|_{p}}\text{ , }\langle\sigma^+(p)\mid\sigma^-(p)\rangle=\frac{2}{\|h(p)v_+\|_{p}\|h(p)v_-\|_{p}},\]
		and for $a=\sqrt{\|h(p)v_-\|_{p}^t/\|h(p)v_+\|_{p}^t}$ we deduce that	
		\[\frac{\sigma^\pm(p)}{\sqrt{2\langle\sigma^+(p)\mid\sigma^-(p)\rangle}}=\frac{1}{2}a^{\pm1}h(p) v_\pm.\]
		Now we can choose $h_a\in\G$ such that $h_av_\pm=a^{\pm1}v_\pm$ and consider $g(p)=h(p)h_a$ to conclude our result.
	\end{proof}
	\begin{remark} \label{rem:bh}
		As $\s_{\G}(e_{2n})=\Go$, we obtain that $\bH\defeq\G e_{2n}\cong\G/\Go$. We define
		\[\sigma:=\frac{(\sigma^+-\sigma^-)}{\sqrt{2\langle\sigma^+\mid\sigma^-\rangle}}\]
		and observe that $\sigma:\cflow\to\bH$. As $\sigma^\pm$ are $\Gamma$-equivariant, we deduce that $\sigma(\gamma p)=\rho(\gamma)\sigma(p)$ for all $p\in\cflow$.  
	\end{remark}
	
	Suppose $s\in\R$, $\gamma\in\Gamma$ and $(p,x)\in\cflow\times\bH$. We observe that $\R$ acts on $\cflow\times\bH$ by sending $(p,x)$ to $(\phi_sp,x)$ and $\Gamma$ acts on $\cflow\times\bH$ by sending $(p,x)$ to $(\gamma p,\rho(\gamma)x)$.

	\begin{lemma}\label{lem:glm}
		The $\Gamma$ action on $(\cflow\times\bH)/\R$ is proper if and only if the $\R$ action on $\Gamma\backslash(\cflow\times\bH)$ is proper.
	\end{lemma}
	\begin{proof}
		As the action of $\Gamma$ and the action of $\R$ on $\cflow$ commute with each other, we see that 
		\begin{align*}
			\gamma\phi_t(p,x)&=\gamma(\phi_tp,x)=(\gamma \phi_tp,\rho(\gamma)x)\\
			&=(\phi_t\gamma p,\rho(\gamma)x)=\phi_t(\gamma p,\rho(\gamma)x)=\phi_t\gamma(p,x).
		\end{align*}
		Now we use Lemma 5.2 of \cite{GLM} (See also Lemma 3.1 of \cite{Ben}) to conclude our result.
	\end{proof}
	\begin{lemma}[Goldman--Labourie \cite{GL}]\label{lem:gl3}
		Suppose $f:\flow\to\R$ is a H\"older continuous function such that
		\[\int f d\mu>0\]
		for all $\phi$ invariant measures $\mu$ on $\flow$. Then $f$ is Liv\v sic cohomologous to a positive function.
	\end{lemma}
	\begin{proof}
		The proof follows verbatim the proof given in the proof of Lemma 3 of \cite{GL} once we replace all the appearances of a manifold by the compact metric space $\flow$.
	\end{proof}
	
	\begin{proposition}\label{prop:p1}
		Suppose $\rho\in\sHom(\Gamma,\G,\sP^\pm)$ and the action of $\rho(\Gamma)$ on $\bH$ is proper. Then the proto-Labourie-Margulis invariants of $\rho$ are non-vanishing.
	\end{proposition}
	
	\begin{proof}
		Let us assume on contrary that the action of $\rho(\Gamma)$ on $\bH$ is proper but the proto-Labourie-Margulis invariants $f^\pm$ of $\rho$ are not non-vanishing. Then neither $f^+$ nor $f^-$ is Liv\v sic cohomologous to a strictly positive function. Hence using Lemma \ref{lem:gl3} and the fact that $(f^++f^-)/2$ is Liv\v sic cohomologous to $0$, we get that there exists a flow invariant probability measure $\mu$ such that 
		\[\int \frac{1}{2}(f^+-f^-)d\mu=0.\] 
		Suppose $T>0$. We consider 
		\[f_T(p)\defeq\frac{1}{T}\int_0^T\frac{1}{2}(f^+-f^-)(\phi_sp)ds\] 
		and observe that $\int f_Td\mu=0$. Hence for all $T>0$ there exists $p_T\in\flow$ such that $f_T(p_T)=0$. It follows that
		\[\log\left(\frac{\|\nu^+(p_T)\|_{p_T}}{\|\nu^+(p_T)\|_{\phi_Tp_T}}\right)-\log\left(\frac{\|\nu^-(p_T)\|_{p_T}}{\|\nu^-(p_T)\|_{\phi_Tp_T}}\right)=0.\]
		We use Lemma \ref{lem:proto} to obtain that 
		\[\frac{\|\nu^\pm(p_T)\|_{p_T}}{\|\nu^\pm(p_T)\|_{\phi_Tp_T}}=\frac{\|h(p_T)v_\pm\|_{p_T}}{\|h(p_T)v_\pm\|_{\phi_Tp_T}}\]
		and hence for all $T>0$, we have $\sigma_\rho(\phi_Tp_T)=\sigma_\rho(p_T)$.
		Moreover, $\flow$ is compact. It follows that $\R$ does not act properly on $\Gamma\backslash(\cflow\times\bH)$. Now we use Lemma \ref{lem:glm} to get that     $\Gamma$ does not act properly on $(\cflow\times\bH)/\R=\bdry^{(2)}\times\bH$ and hence $\Gamma$ does not act properly on $\bH$ a contradiction.
	\end{proof}
	\begin{proposition}\label{prop:p2}
		If $\rho\in\sHom(\Gamma,\G,\sP^\pm)$ is such that the proto-Labourie-Margulis invariants of $\rho$ are non-vanishing. Then the action of $\rho(\Gamma)$ on $\bH$ is proper.
	\end{proposition}
	\begin{proof}
		Suppose the action of $\rho(\Gamma)$ on $\bH$ is not proper. Hence, there exist $\gamma_k\in\Gamma$ going to infinity and $x_k\in\bH$ such that the sequence $\{x_k\}$ converge to some $x\in\bH$ and $\{y_k:=\rho(\gamma_k)x_k\}$ converge to some $y\in\bH$. As $\rho(\Gamma)$ is Anosov with respect to $\sP^\pm$, we use Theorem 1.7 of \cite{GW2} and Remark \ref{rem:nonparano} to get that $\rho(\Gamma)$ is AMS proximal. Hence without loss of generality we can assume that $\gamma_k$ is of infinite order for all $k$ and $\lim_{k\to\infty}\gamma^+_k\neq\lim_{k\to\infty}\gamma^-_k$ (please see Section 8.1 of \cite{GT} for a more detailed version of this argument). Therefore, we obtain that the flowlines corresponding to $\gamma_k$ converge to a flowline with boundary points $\lim_{k\to\infty}\gamma^\pm_k$ and
		\[\lim_{k\to\infty}l(\gamma_k)=\infty.\]
		We choose $q_k\in\cflow$ such that $q_k$ is a point on the flowline corresponding to $\gamma_k$ and $q_k$ converges to some $q\in\cflow$. For $z\in\{x,y\}$ suppose  $w^\pm_{z,k}\in\xi^\pm(\gamma_k^\pm)$ and $v^\pm_{z,k}\in\R\nu^\pm(q_k)$ are such that 
		\[z_{k}=w^+_{z,k}+w^-_{z,k}+v^+_{z,k}+v^-_{z,k}.\] 
		As $\lim_{k\to\infty} z_k=z$, we obtain that $\lim_{k\to\infty} w^\pm_{z,k}=w_z^\pm$ and $\lim_{k\to\infty}v^\pm_{z,k}=v_z^\pm$ for some finite $w_z^\pm$ and $v_z^\pm$. Hence, for $z\in\{x,y\}$ it follows that, 
		\begin{align*}
			\langle z\mid \nu^+(q)\rangle\langle z\mid \nu^-(q)\rangle&=\langle v_z^+\mid v_z^-\rangle\langle \nu^+(q)\mid \nu^-(q)\rangle,\\
			\langle w_z^+\mid w_z^-\rangle + \langle v_z^+\mid v_z^-\rangle&=\langle z\mid z\rangle=-1.
		\end{align*}
		Moreover, as $\rho$ is $\sP^\pm$-Anosov, we deduce that $w_x^+=0$ and $w_y^-=0$. Hence,
		\[\langle x\mid \nu^+(q)\rangle\langle x\mid \nu^-(q)\rangle=-\langle \nu^+(q)\mid \nu^-(q)\rangle=\langle y\mid \nu^+(q)\rangle\langle y\mid \nu^-(q)\rangle.\]
		We also observe that
		\begin{align*}
			\lim_{k\to\infty}\lambda^{\pm}(\gamma_k)&=\lim_{k\to\infty}\frac{\langle x_k\mid \nu^\pm(q_k)\rangle}{\langle x_k\mid \rho(\gamma_k)^{-1}\nu^\pm(q_k)\rangle}\\
			&=\lim_{k\to\infty}\frac{\langle x_k\mid \nu^\pm(q_k)\rangle}{\langle \rho(\gamma_k)x_k\mid \nu^\pm(q_k)\rangle}=\frac{\langle x\mid \nu^\pm(q)\rangle}{\langle y\mid \nu^\pm(q)\rangle}.
		\end{align*}
		Therefore, we deduce that $\lim_{k\to\infty}\lambda^\pm(\gamma_k)$ are finite nonzero numbers. Suppose $f^\pm$ are the proto-Labourie-Margulis invariants of $\rho$. As $\lim_{k\to\infty}l(\gamma_k)=\infty$, it follows that
		\[\lim_{k\to\infty}\int_{\gamma_k} f^+ =\lim_{k\to\infty}\frac{\log\lambda^+(\gamma_k)}{l(\gamma_k)} =0.\] 
		We also know that the space of flow invariant probability measures on $\flow$ is weak* compact.  Hence, there exists a flow invariant probability measure $\mu$ on $\flow$ such that $\int f^+d\mu=0$. It follows that $f^+$ is not Liv\v sic cohomologous to any non-vanishing function and we conclude our result. 
	\end{proof}
	
	\subsection{Special Anosov representations}
	In this subsection we relate the existence of non-vanishing proto-Labourie-Margulis invariants with certain special Anosov representations in $\G$.
	\begin{notation}
		Suppose $f:\flow\to\R$. We define,
		\[I_{a,b}(f)(p):=\frac{1}{b-a}\int_a^bf(\phi_sp)ds\]
	\end{notation}
	\begin{lemma}[Goldman--Labourie \cite{GL}]\label{lem:gl7}
		Let $f:\flow\to\R$ be a H\" older continuous function which is differentiable along the flow lines of $\phi$ and 
		\[\int f d\mu>0\]
		for all $\phi$ invariant measure $\mu$. Then $I_{0,T}(f)>0$ for some $T>0$.
	\end{lemma}
	\begin{proof}
		The proof follows verbatim the proof given in the proof of Lemma 7 of \cite{GL} once we replace all the appearances of a manifold by the compact metric space $\flow$.
	\end{proof}
	
	\begin{lemma}\label{lem:fano}
		Suppose $\rho\in\sHom(\Gamma,\G,\sP^\pm)$ and $[f^\pm]$ are the proto-Labourie-Margulis invariants of $\rho$. Also, suppose $f^+$ is Liv\v sic cohomologous to a strictly positive function. Then there exist positive constants $C$ and $k$ such that for all $t>0$ and $p\in\cflow$ the following hold:
		\begin{align*}
			\|\nu^+(p)\|_{\phi_tp}&\leqslant C\exp(-kt)\|\nu^+(p)\|_p,\\
			\|\nu^-(p)\|_{\phi_{-t}p}&\leqslant C\exp(-kt)\|\nu^-(p)\|_p.
		\end{align*}
	\end{lemma}
	\begin{proof}
		Suppose $f^+$ is Liv\v sic cohomologous to a strictly positive function. Hence $-f^-$ is also Liv\v sic cohomologous to a strictly positive function. Then using Lemma \ref{lem:gl7} we get that there exists $c>0$ such that $I_{0,c}(f^+)>0$. Now as $\flow$ is compact we get that $I_{0,c}(f^+)>k^+>0$ for some $k^+$. Hence for $t>0$ and integer $n$ such that $(n+1)>(t/c)\geqslant n$ we have
		\[tI_{0,t}(f^+)(p)\geqslant(t-nc)I_{0,t-nc}(f^+)(p)+nck^+\]
		As both $[0,c]$ and $\flow$ are compact, we obtain that
		\[\min\{I_{0,t}(f^+)(p)\mid (t,p)\in[0,c]\times\flow\}> k^\prime\text{ for some } k^\prime<0.\]
		As $(t-nc)\in[0,c]$, we have $I_{0,t-nc}(f^+)(p)\geqslant k^\prime$.
		Hence, 
		\[tI_{0,t}(f^+)(p)\geqslant (t-nc)k^\prime+nck^+.\]
		As $k^\prime<0$, we obtain that $(t-nc)k^\prime>ck^\prime$. We denote $\exp(k^+c-k^\prime c)$ by $C^+$ and recalling the definition of $f^+$ deduce that 
		\[\|\nu^+(p)\|_{\phi_tp}\leqslant C^+\exp(-k^+t)\|\nu^+(p)\|_p\]
		for all $t>0$ and for all $p\in\cflow$.	As $-f^-$ is also Liv\v sic cohomologous to a strictly positive function, we can do a similar computation to obtain positive constants $C^-$ and $k^-$ such that 
		\[\|\nu^-(p)\|_{\phi_{-t}p}\leqslant C^-\exp(-k^-t)\|\nu^-(p)\|_p\]
		for all $t>0$ and for all $p\in\cflow$.	Finally, we choose $C\defeq\max\{C^+,C^-\}$, $k\defeq\min\{k^+,k^-\}$ and conclude our result.
	\end{proof}
	
	\begin{notation}
		We denote $\mathsf{SL}(2n,\R)$ by $\Gl$ and the stabilizer of the oriented $n$-dimensional subspace $V_\pm$ inside $\Gl$ by $\sQ^\pm$ i.e.
		\[\sQ^\pm:=\s_{\Gl}\left(\oV_\pm\right).\]
		Also, we consider the orientation on $\R\nu^\pm(p)$ coming from the direction of the vector $\nu^\pm(p)$ to be positive and denote this oriented line by $\overrightarrow{\R}\nu^\pm(p)$.
	\end{notation}
	
	\begin{proposition}\label{prop:meano1}
		Suppose $\rho\in\sHom(\Gamma,\G,\sP^\pm)$ and the left proto-Labourie-Margulis invariant $[f^+]$ is positive. Then $\rho\in\sHom(\Gamma,\Gl,\sQ^\pm)$. 
	\end{proposition}

	\begin{proof}
		Suppose $p=(p_+,p_-,t)\in\cflow$ and $\xi^\pm:\bdry\to\G/\sP^\pm$ are the limit maps of $\rho\in\sHom(\Gamma,\G,\sP^\pm)$. We define 	 
		\begin{align*}
			\eta^\pm:\bdry&\longrightarrow\Gl/\sQ^\pm\\
			p_\pm&\longmapsto \xi^\pm(p_\pm)\oplus\oR\nu^\pm(p).
		\end{align*}
		and let $\eta(p):=(\eta^+(p_+),\eta^-(p_-))$. We use Lemmas \ref{lem:protoneu} and \ref{lem:nu0} to observe that $\eta^\pm$ is well defined and 
		\[\eta(\cflow)\subset \Gl/(\sQ^+\cap\sQ^-).\]
		Moreover, using properties of $\nu^\pm$ we obtain that $\eta^\pm$ are $\rho(\Gamma)$-equivariant. Therefore, to show that $\rho\in\sHom(\Gamma,\Gl,\sQ^\pm)$ we only need to produce equivariant metrics with contraction properties on: 
		\[\sT_{\eta(p)}(\Gl/(\sQ^+\cap\sQ^-))\cong\sT_{\eta^+(p_+)}(\Gl/\sQ^+)\oplus\sT_{\eta^+(p_+)}(\Gl/\sQ^+).\]
		As the underlying subspaces of $\eta^\pm(p_\pm)$ add up to $\R^{2n}$, we use Proposition 10.1 of \cite{Voi1} to obtain that
		\[\sT_{\eta^\pm(p_\pm)}(\Gl/\sQ^\pm)\cong\sHom(\eta^\pm(p_\pm),\eta^\mp(p_\mp)).\]
		We use Corollary 3.3 of \cite{GT} and Lemma \ref{lem:fano} to obtain equivariant metrics with contraction properties on $\eta^\pm(p_\pm)=\xi^\pm(p_\pm)\oplus\oR\nu^\pm(p)$. Indeed, we use the following recipe: for $v\in\xi^\pm(p_\pm)$ and $w\in\R\nu^\pm(p)$ we define,
		\[\|(v,w)\|_p^2:=\|v\|_p^2+\|w\|_p^2.\] 
		Now for $A\in\sHom(\eta^\pm(p_\pm),\eta^\mp(p_\mp))$ we define
		\[\|A\|_p=\sup_{v\neq0}\frac{\|A(v,w)\|_p}{\|(v,w)\|_p}.\]
		As $\rho\in\sHom(\Gamma,\G,\sP^\pm)$ and the left proto-Labourie-Margulis invariant is positive, we obtain that $\|(v,w)\|_{\phi_{\pm t}p}\leqslant Ce^{-kt}\|(v,w)\|_p$ and $\|A(v,w)\|_{\phi_{\mp t}p}\leqslant Ce^{-kt}\|A(v,w)\|_p$ for all $t>0$. Hence, for all $v\neq0$ we deduce that
		\[\frac{\|A(v,w)\|_{\phi_{\mp t}p}}{\|(v,w)\|_{\phi_{\mp t}p}}\leqslant C^2e^{-2kt}\frac{\|A(v,w)\|_p}{\|(v,w)\|_p}\leqslant C^2e^{-2kt}\|A\|_p\]
		and take the supremum over all  $v\neq0$ in the left hand side to conclude that $\|A\|_{\phi_{\mp t}p}\leqslant C^2e^{-2kt}\|A\|_p$ for all $t>0$.
	\end{proof}
	
	\begin{proposition}\label{prop:meano2}
		Suppose $\rho\in\sHom(\Gamma,\G,\sP^\pm)\cap\sHom(\Gamma,\Gl,\sQ^\pm)$. Then the left proto-Labourie-Margulis invariant $[f^+]$ is positive. 
	\end{proposition}
	
	\begin{proof}
		Suppose $\xi^\pm:\bdry\to\G/\sP^\pm$ and $\eta^\pm:\bdry\to\Gl/\sQ^\pm$ are the two pairs of limit maps.
		As $\rho(\Gamma)\subset\G$ and the action of $\rho(\gamma)$ (resp. $\rho(\gamma)^{-1}$) is contracting on $\eta^+(\gamma^+)$ (resp. $\eta^-(\gamma^-)$), we obtain that $\eta^\pm(\gamma^\pm)$ are isotropic subspaces of $R^{n,n}$. As $\eta^\pm(\gamma^\pm)$ are of dimension $n$, it follows that they are maximally isotropic. Hence by continuity of the limit maps we deduce that the image of $\eta^\pm$ lies inside the space of all maximal isotropic subspaces of $\R^{n,n}$. Also, as the action of $\rho(\gamma)$ (resp. $\rho(\gamma)^{-1}$) is contracting on $\xi^+(\gamma^+)$ (resp. $\xi^-(\gamma^-)$), we obtain that $\xi^\pm(\gamma^\pm)\subset\eta^\pm(\gamma^\pm)$.
		Now using the orientation on $\eta^\pm(p_\pm)$ and the orientations already on $\xi^\pm(p_\pm)$ we get positive orientations on $\eta^\pm(p_\pm)\cap(\xi^\mp(p_\mp))^\perp$.
		As $\eta^\pm(\gamma^\pm)$ are respectively in the orbits of $V_\pm$, we use Lemmas \ref{lem:maxorient} and \ref{lem:choice} to obtain that
		\[\eta^\pm(p_\pm)\cap(\xi^\mp(p_\mp))^\perp=\oR\nu^\pm(p).\]
		
		Moreover, we know that the contraction property of an Anosov representation does not depend on a particular choice of norms up to H\"older equivalence. Hence we can choose the collection of norms to be smooth along flow lines (please see Remark \ref{rem:norm}) and the contraction property would still hold. Now using Lemma 5.3 of \cite{BCLS} we can choose $C=1$. Hence, we get that there exists a positive constant $k$ and a collection $\{\|\cdot\|_p\mid p\in\cflow\}$ of Euclidean norms on $\R^{2n}$ such that:
		\begin{enumerate}
			\item it is H\"older continuous in the variable $p\in\cflow$,
			\item it is smooth along the flow lines of $\{\phi_s\}_{s\in\R}$,
			\item it is equivariant i.e. $\|\rho(\gamma)v\|_{\gamma p}=\|v\|_p$ for all $v\in\R^{2n}$ and $\gamma\in\Gamma$,
			\item it is contracting i.e for all $p\in\cflow$, $v\in\eta^+(p_+)$, $w\in\eta^-(p_-)$ and $s>0$:
			\[\frac{\|v\|_{\phi_sp}}{\|w\|_{\phi_sp}}\leqslant \exp(-ks)\frac{\|v\|_p}{\|w\|_p}.\]
		\end{enumerate}
		Therefore, for all $p\in\cflow$ we obtain
		\begin{align*}
			g(p)\defeq\left.\frac{\partial}{\partial s}\right|_{s=0}\log\frac{\|\nu^+(p)\|_{\phi_sp}}{\|\nu^-(p)\|_{\phi_sp}}&=\lim_{s\to 0} \frac{1}{s}\log\left(\frac{\|\nu^+(p)\|_{\phi_sp}}{\|\nu^-(p)\|_{\phi_sp}}\frac{\|\nu^-(p)\|_p}{\|\nu^+(p)\|_p}\right)\\
			&\leqslant \lim_{s\to 0} \frac{1}{s}\log\left(\exp(-ks)\right)= -k.
		\end{align*}
		Hence for all flow invariant probability measure $\mu$ on $\flow$ we get that
		\[\int\frac{1}{2}(f^+-f^-)d\mu=-\int g d\mu \geqslant k >0.\]
		Finally, using Remark \ref{rem:invcohom} we conclude our result.
	\end{proof}
	\begin{remark}\label{rem:unorien}
		Suppose $\rho\in\sHom(\Gamma,\Gl)$. We observe that 
		\[J\sQ^\pm J=\s_\G(J\oV_\pm).\]
		\begin{enumerate}
			\item If $\rho\in\sHom(\Gamma,\G,\sP^\pm)$, then we use Lemmas \ref{lem:maxorient} and \ref{lem:choice} to deduce that the left proto-Labourie-Margulis invariant $[f^+]$ is negative if and only if $\rho\in\sHom(\Gamma,\Gl,J\sQ^\pm J)$.
			\item If $\rho(\Gamma)\subset\G$, then by Lemma \ref{lem:maxorient} we deduce that $\rho\in\sHom(\Gamma,\Gl,\sQ^\pm)$ (resp. $\sHom(\Gamma,\Gl,J\sQ^\pm J)$) if and only if $\rho$ is Anosov with respect to the stabilizers of $V_\pm$ (resp. $JV_\pm$). 
			\item If $\rho$ is Anosov in $\Gl$ with respect to the stabilizers of a pair of transverse $n$ dimensional subspaces and $\rho(\Gamma)\subset\G$, then we use Lemma \ref{lem:choice} to deduce that either $\rho$ is Anosov with respect to the stabilizers of $V_\pm$ or $\rho$ is Anosov with respect to the stabilizers of $JV_\pm$.
		\end{enumerate}
	\end{remark}
	
	\begin{theorem}\label{thm:p}
		Suppose $\rho\in\sHom(\Gamma,\G,\sP^\pm)$. Then the action of $\rho(\Gamma)$ on $\bH$ is proper if and only if $\rho$ is Anosov in $\Gl$ with respect to the stabilizer of an $n$-dimensional subspace. 
	\end{theorem}
	\begin{proof}
		We use Propositions \ref{prop:p1}, \ref{prop:p2}, \ref{prop:meano1}, \ref{prop:meano2} and Remark \ref{rem:unorien} to conclude our result.
	\end{proof}

	\section{Infinitesimal proper actions}
	In this section we show that proper affine actions on $\R^{n,n-1}$ can be seen as infinitesimal versions of proper actions on $\bH$. We do this by recalling the notion of affine Anosov representations from \cite{GT} and relating them with the special Anosov representations appearing in Section \ref{sec:specialAnosov}.
	
	\subsection{Affine Anosov representations} \label{sec:aar}
	
	In this subsection we recall the notion of an affine Anosov representation. Suppose $\Ga:=\Go\ltimes\R^{2n-1}$. We call an element $h\in\Go$ \emph{pseudo-hyperbolic} if the unit eigenspace of $h$ is one dimensional and $h$ does not have $-1$ as an eigenvalue. We call an element $g=(h,u)\in\Ga$ \emph{pseudo-hyperbolic} if its linear part $h$ is pseudo-hyperbolic.
	
	We call an $n$ dimensional subspace $V\subset\bR$ a \textit{null space} if $(V^\perp)\cap\R^{2n-1}$ is a maximal isotropic subspace of $\bR$. Affine subspaces which are parallel to null spaces are called \textit{affine null spaces}. We recall Lemmas \ref{lem:stiso} and \ref{lem:maxorient} and note that we can consistently provide the maximal isotropic spaces and the null spaces with positive orientations. 
	
	Now suppose $g=(h,u)\in\Ga$ is such that $h$ is pseudo-hyperbolic. Let
	\[W^h_\pm:=\left\{v\mid \lim_{k\to\infty}h^{\mp k}v=0\right\}\subset \R^{2n-1}.\]
	We note that $W^h_\pm$ are maximal isotropic subspaces of $\bR$ and let $v^h_0$ be the unique eigenvector of $h$ with eigenvalue $1$ such that $\langle v_0^h\mid v_0^h\rangle=1$ and which is positively oriented with respect to the orientations on $W^h_+$ and $(W^h_+)^\perp\cap\bR$ (for more details please see \cite{AMS}). Then the \emph{Margulis invariant} of $g$ is defined as:
	\[\alpha(g):=\langle u\mid v^h_0\rangle .\]
	
	Let $\Gamma$ be a word hyperbolic group and let $(\rho,{u}):\Gamma\rightarrow\Ga$
	be an injective homomorphism such that the linear part $\rho$ is $\sP_0^\pm$-Anosov with limit maps given by 
	\[\xi_\rho^\pm:\bdry\rightarrow\Go/\sP_0^\pm\] 
	and $\xi_\rho(p)=(\xi_\rho^+(p_+),\xi_\rho^-(p_-))$ for all $p=(p_+,p_-,t)\in\cflow$.
	\begin{remark}\label{rem:nu}
		Suppose $(\rho,{u}):\Gamma\rightarrow\Ga$ is an injective homomorphism such that $\rho\in\sHom(\Gamma,\G,\sP_0^\pm)$ and $\gamma^\pm\in\bdry$ are respectively the attracting and repelling points of the action of any infinite order element $\gamma\in\Gamma$. We observe that $\xi_\rho(\gamma^+,\gamma^-,t)$ is independent of $t$. Henceforth, we denote $\nu(\xi_\rho(\gamma^+,\gamma^-,t))$ by $\nu_\rho(\gamma^+,\gamma^-)$ and observe that $\nu_\rho(\gamma^+,\gamma^-)=v_0^{\rho(\gamma)}$. It follows that
		\[\alpha((\rho,u)(\gamma))=\langle {u}_\rho(\gamma)\mid\nu_\rho(\gamma^+,\gamma^-)\rangle.\]
	\end{remark}
	
	\begin{definition}\label{def:LMinv}
		Let $\Gamma$ be a word hyperbolic group and let $(\rho,{u})$ be an injective homomorphism from $\Gamma$ to $\Ga$ such that $\rho\in\sHom(\Gamma,\Go,\sP_0^\pm)$. Then the \emph{Labourie-Margulis invariant} of this representation is a Liv\v sic cohomology class $[f_{\rho,u}]$ of H\"older continuous functions $f_{\rho,u}$ such that
		\[\int f_{\rho,u}d\mu_\gamma=\frac{\alpha(\rho,u)(\gamma)}{l(\gamma)}\]
		where $\mu_\gamma$ is a flow invariant probability measure supported on the periodic orbit of $\flow$ corresponding to $\gamma$ and $l(\gamma)$ is the period of this orbit.
	\end{definition}
	\begin{remark}
		The existence of the Labourie-Margulis invariants follow from the constructions in Appendix 8.2 and Lemma 7.2 of \cite{GT}. Moreover, the uniqueness follows from Liv\v sic's theorem \cite{Liv}.
	\end{remark}
	Now we define the notion of affine Anosov representations. Suppose $W_\pm\subset\R^{2n-1}$ be as in the previous section. We observe that 
	\[\s_{\Go}(W_\pm)=\s_{\Go}(W_\pm^\perp\cap\R^{2n-1}).\]
	Henceforth, in this section we treat $(W_\pm^\perp\cap\R^{2n-1})$ as affine subspaces in $\R^{2n-1}$ and call the stabilizers,
	\[\sP_a^\pm\defeq\s_{\Ga}(W_\pm^\perp\cap\R^{2n-1}),\] 
	of these subspaces under the action of the affine group $\Ga$ as \emph{pseudo parabolic} subgroups. These subgroups of $\Ga$ are used in the definition of an affine Anosov representation in the same way as parabolic subgroups are used in the definition of an Anosov representation. We observe that for $\sL_0$ and $e_n$ as defined in the previous section we have \[\sL_a\defeq\sP_a^+\cap\sP_a^-=\sL_0\ltimes\R e_n.\] 
	Let $\cX_a$ be the space of all affine null subspaces in $\R^{2n-1}$ and let $\cY_a$ be the space of all transverse pairs of affine null subspaces. Then $\cY_a$ is an open and dense subset of $\cX_a\times\cX_a$. The group $\Ga$ acts transitively on the space $\cX_a$ and we have $\cX_a\cong\Ga/\sP_a^\pm$. Moreover, the diagonal action of $\Ga$ is transitive on $\cY_a$ and we have $\cY_a\cong\Ga/\sL_a$.
	\begin{definition}
		Let $(\rho,{u}):\Gamma\rightarrow\Ga$ be an injective homomorphism. Then $(\rho,{u})$ is called affine Anosov with respect to $\sP_a^\pm$ if and only if the following conditions hold:
		\begin{enumerate}
			\item \label{cond:1}
			\begin{enumerate}
				\item There exist a continuous, injective, $(\rho,{u})(\Gamma)$-equivariant limit maps
				$\xi^\pm:\bdry\rightarrow\Ga/\sP_a^\pm$	such that $\xi(p)\defeq(\xi^+(p_+),\xi^-(p_-))\in\Ga/\sL_a$ for $p=(p_+,p_-,t)\in\cflow$ .
				\item There exist positive constants $C,c$ and for $p\in\cflow$ a continuous collection of $(\rho,{u})(\Gamma)$-equivariant Euclidean metrics $\|\cdot\|_p$ on $\sT_{\xi(p)}(\Ga/\sL_a)$ such that for all $v^\pm\in\sT_{\xi^\pm(p_\pm)}(\Ga/\sP_a^\pm)$ and $t\geqslant 0$:
				\[\|v^\pm\|_{\phi_{\pm t}p}\leqslant Ce^{-ct}\|v^\pm\|_p.\]
			\end{enumerate}
			\item \label{cond:2} There exists a $(\rho,{u})(\Gamma)$-equivariant map $s:\cflow\rightarrow\R^{2n-1}$ which is H\"older continuous and is differentiable along the flow lines of $\phi$. Moreover, for all $p\in\cflow$ the function
			\[f(p)\defeq\left\langle\left.\frac{\partial}{\partial t}\right|_{t=o}s(\phi_tp)\mid\nu_\rho(p)\right\rangle \neq 0.\]
		\end{enumerate}
	\end{definition}
	\begin{remark}
		We note that whenever the first condition of the above definition is satisfied one can use a partition of unity type arguement to guarantee the existence of a $(\rho,{u})(\Gamma)$-equivariant map $s:\cflow\rightarrow\R^{2n-1}$ which is H\"older continuous and is differentiable along the flow lines of $\phi$ (for more details please see the Appendix 8.2 of \cite{GT}).
	\end{remark}
	We denote the space of all representations in $\Ga$ which are affine Anosov with respect to $\sP^\pm_a$ by $\sHom(\Gamma,\Ga,\sP_a^\pm)$.
	\begin{remark}
		Suppose $(\rho,u)\in\sHom(\Gamma,\Ga)$. Then we note that $(\rho,u)$ satisfy Condition \ref{cond:1} in the above definition if and only if $\rho\in\sHom(\Gamma,
		\Go,\sP_0^\pm)$. Moreover, Condition \ref{cond:2} in the above definition is equivalent to saying that the Labourie--Margulis invariant of $(\rho,u)$ is non-vanishing.
	\end{remark}
	We state the key property of affine Anosov representations:
	\begin{theorem}[Ghosh--Treib \cite{GT}]\label{thm:GT}
		Suppose $\rho\in\sHom(\Gamma,\Go,\sP_0^\pm)$ and $(\rho,u)\in\sHom(\Gamma,\G_a)$. Then $(\rho,u)\in\sHom(\Gamma,\Ga,\sP_a^\pm)$ if and only if the action of $(\rho,u)(\Gamma)$ on $\R^{2n-1}$ is proper.
	\end{theorem}
	It is important to mention here the existence and non-existence results due to Abels--Margulis--Soifer \cite{AMS}:
	\begin{theorem}[Abels--Margulis--Soifer \cite{AMS}]
		The following holds:
		\begin{enumerate}
			\item There exist free subgroups of $\mathsf{Aff}(2n-1,\R)$ with linear part Zariski dense in $\mathsf{SO}(n,n-1)$ which act properly discontinuously on $\R^{2n-1}$, when $n$ is even.
			\item There does not exist any subgroup of $\mathsf{Aff}(2n-1,\R)$ with linear part Zariski dense in $\mathsf{SO}(n,n-1)$ which acts properly discontinuously on $\R^{2n-1}$, when $n$ is odd.
		\end{enumerate}
	\end{theorem}

	\subsection{Margulis invariants as derivatives}
	
	In this subsection we relate Margulis invariants with derivatives of certain eigenvalues. 
	
	\begin{remark}
		Suppose $\fg$ (resp. $\fg_0$) denote the Lie algebra of $\G$ (resp. $\Go$) and $G\in\fg$. We observe that $G\in\fg_0$ if and only if $Ge_{2n}=0$.
		
		As $\langle ge_{2n}\mid ge_{2n}\rangle$ is constant, we obtain that $\langle Ge_{2n}\mid e_{2n}\rangle=0$ for all $G\in\fg$. Hence, $Ge_{2n}\in\bR$.	Moreover, for $v\in\bR\subset\R^{n,n}$ we consider
		\[G=\begin{bmatrix}
			0&v\\
			v^tI_{n,n-1}&0
		\end{bmatrix}\]
		and observe that $G\in\fg$ with $Ge_{2n}=v$.	
	\end{remark}	
	
	\begin{remark}\label{rem:protonu}
		Suppose $g\in\G_0$ is a pseudo-hyperbolic element and $W_\pm^g$ be the attracting and repelling subspaces of $g$ (see Section \ref{sec:aar}). Observe that the dimensions of $W_\pm^g$ are $(n-1)$. Also, suppose $g_t\in\G$ is an analytic one parameter family with $g=g_0$. Then, for $t$ small enough we obtain a pair of attracting and repelling subspaces of $g_t$ denoted by $W_\pm^{g_t}$ and whose dimensions are also $(n-1)$. Moreover, using Lemma \ref{lem:choice} we obtain that there is a unique maximal isotropic space inside $(W_+^{g_t})^\perp$ which is in the orbit of $V_+$ (see Section \ref{sec:proto-neutral}). We denote this space by $V_+^{g_t}$. We observe that $g_t$ preserves the line $(W_-^{g_t})^\perp\cap V_+^{g_t}$. Let $\lambda(g_t)$ be the eigenvalue of the action of $g_t$ on this line. We observe that the definition of $\lambda$ given here is compatible with the definition of $\lambda^+$ given in Notation \ref{not:lam}. 
		
		Finally, we start using the following notation:
		\[d\lambda(g,G):=\left.\frac{d}{dt}\right|_{t=0}\lambda(g_t).\]
	\end{remark}
	
	\begin{lemma}\label{lem:deral}
		Suppose $g\in\G_0$ is a pseudo-hyperbolic element and $g_t\in\G$ is an analytic one parameter family whose tangent direction at $g=g_0$ is $G$ and $Ge_{2n}=v$. Then 
		\[\alpha(g,v)=d\lambda(g,G).\]
	\end{lemma}
	\begin{proof}
		We consider the orientation on $V_\pm^{g_t}$ coming from $\oV_\pm$ and the orientation on $W_\pm^{g_t}$ coming from $\oW_\pm$. Let $v^\pm_t\in(W_\mp^{g_t})^\perp\cap V_\pm^{g_t}$ respectively be the $\|\cdot\|$-unit vectors which is positively oriented with respect to the orientations on $V_\pm^{g_t}$ and $W_\pm^{g_t}$. Hence, 
		\[\lambda(g_t)=\langle g_t^{\pm1}v^\pm_t\mid v^\mp_t\rangle/\langle v_t^+\mid v_t^-\rangle.\]
		We take derivative on both sides and deduce that
		\[\left.\frac{d}{dt}\right|_{t=0}\lambda(g_t)={\langle \pm Gv_0^\pm\mid v_0^\mp\rangle}/{\langle v_0^+\mid v_0^-\rangle}.\]
		We observe that $e_{2n}=(av_0^+-bv_0^-)$ for some $a,b\in\R$. As $\langle e_{2n}\mid e_{2n}\rangle=-1$, we obtain that $2ab\langle v_0^+\mid v_0^-\rangle=1$. It follows that 
		\[d\lambda(g,G)=2ab\langle Gv_0^+\mid v_0^-\rangle.\] 
		Moreover, as $\langle g_tv_t^\pm\mid v_t^\pm\rangle=0$, we obtain that $\langle Gv_0^\pm\mid v_0^\pm\rangle=0$. Therefore,
		\[d\lambda(g,G)=\langle G(av_0^+-bv_0^-)\mid (av_0^++bv_0^-)\rangle=\langle v\mid (av_0^++bv_0^-)\rangle.\]
		Finally, we conclude by observing that $(av_0^++bv_0^-)=v^g_0$.
	\end{proof}
	\begin{remark}
		If $\alpha(g,v)\neq0$, then for $t$ small enough $\lambda(g_t)\neq\lambda^{-1}(g_t)$.
	\end{remark}
	
	\begin{proposition}\label{prop:defano}
		Suppose $\rho: \Gamma \rightarrow \Go$ is Anosov with respect to $\sP_0^\pm$. Then $\rho$ is Anosov in $\G$ with respect to $\sP^\pm$ i.e. $\sHom(\Gamma,\Go,\sP_0^\pm)\subset\sHom(\Gamma,\G,\sP^\pm)$.
	\end{proposition}
	
	\begin{proof}
		Let $\rho: \Gamma \rightarrow \Go$ be Anosov with respect to $\sP_0^\pm$ with limit map
		\[\xi:\cflow\to\Go/\sL_0.\]
		We recall that $\Go/\sL_0\subset\G/\sL$. Hence we get limit maps
		\[\xi:\cflow\to\Go/\sL_0\subset\G/\sL.\]
		Therefore, to show that $\rho\in\sHom(\Gamma,\G,\sP^\pm)$ we need only to show that the contraction properties hold true.
		
		We observe that $\xi^\pm(p_\pm)^\perp=(\xi^\pm(p_\pm)^\perp\cap\R^{2n-1})\oplus\R e_{2n}$,
		\begin{align*}
			\sT_{(\xi^+(p_+),\xi^-(p_-))}(\Go/\sL_0)&=\sT_{\xi^+(p_+)}(\Go/\sP_0^+)\oplus\sT_{\xi^-(p_-)}(\Go/\sP_0^-),\\
			\sT_{(\xi^+(p_+),\xi^-(p_-))}(\G/\sL)&=\sT_{\xi^+(p_+)}(\G/\sP^+)\oplus\sT_{\xi^-(p_-)}(\G/\sP^-),
		\end{align*}
		Hence, using Proposition 10.1 of \cite{Voi1} we obtain that
		\begin{align*}
			\sT_{\xi^\pm(p_\pm)}(\G/\sP^\pm)&\cong\sHom_\textrm{skew}(\xi^\pm(p_\pm),\xi^\mp(p_\mp)^\perp)\\
			&\cong\sHom_\textrm{skew}(\xi^\pm(p_\pm),\xi^\mp(p_\mp)^\perp\cap\R^{2n-1})\oplus\sHom(\xi^\pm(p_\pm),\R e_{2n})\\
			&\cong\sT_{\xi^\pm(p_\pm)}(\Go/\sP_0^\pm)\oplus\sHom(\xi^\pm(p_\pm),\R e_{2n}).
		\end{align*}
		As $\rho\in\sHom(\Gamma,\Go,\sP_0^\pm)$, we have a collection of norms satisfying contraction properties on $\sT_{\xi^\pm(p_\pm)}(\Go/\sP_0^\pm)$. We use Corollary 3.3 of \cite{GT} to obtain a norm satisfying contraction properties on $\sHom(\xi^\pm(p_\pm),\R e_{2n})$.
		Finally, we obtain equivariant norms satisfying contraction properties on $\sT_{\xi^\pm(p_\pm)}(\G/\sP^\pm)$ by using the following recipe:  
		\[\|(v,w)\|^2:=\|v\|^2+\|w\|^2\] 
		where $v\in\sT_{\xi^\pm(p_\pm)}(\Go/\sP_0^\pm)$ and $w\in\sHom(\xi^\pm(p_\pm),\R v_0)$.
	\end{proof}
	
	Let $\{\rho_t\}_{t\in(-1,1)}\subset\sHom(\Gamma,\G,\sP^\pm)$ be an analytic one parameter family and ${U}:\Gamma\to\fg$ be such that
	$\rho_0=\rho$ and for all $\gamma\in\Gamma$,
	\[{U}(\gamma)=\left.\frac{d}{dt}\right|_{t=0}\rho_t(\gamma)\rho(\gamma)^{-1}.\]
	We call ${U}$ a \emph{tangent vector} of $\sHom(\Gamma,\G,\sP^\pm)$ at $\rho$. We observe that 
	\[U\in\sZ^1_{\sad\circ\rho}(\Gamma,\fg):=\{{V}\mid {V}(\gamma\eta)=\sad(\rho(\gamma)){V}(\eta)+{V}(\gamma)\text{ for all } \gamma,\eta\in\Gamma\}.\]
	We consider, ${u}(\gamma):={U}(\gamma)e_{2n}$ for all $\gamma\in\Gamma$ and observe that
	\[u\in\sZ^1_{\rho}(\Gamma,\R^{2n-1}):=\{{v}\mid {v}(\gamma\eta)=\rho(\gamma){V}(\eta)+{V}(\gamma)\text{ for all } \gamma,\eta\in\Gamma\}.\]
	
	\begin{remark}\label{rem:1}
		We observe that for all $\gamma\in\Gamma$ and $p_\gamma=(\gamma^+,\gamma^-,0)$, the action of $\rho(\gamma)$ fixes both $\nu_0^\pm(p_\gamma)$. Hence $\lambda^\pm_0(\gamma)=1$ for all $\gamma\in\Gamma$.
	\end{remark}
	\begin{proposition}\label{prop:deral}
		Suppose $\rho\in\sHom(\Gamma,\Go,\sP_0^\pm)\subset\sHom(\Gamma,\G,\sP^\pm)$ and $U\in\sZ^1_{\sad\circ\rho}(\Gamma,\fg)$ is a tangent vector to $\sHom(\Gamma,\G,\sP^\pm)$ at $\rho$. Then for ${u}={U}e_{2n}$,
		\[\alpha(\rho,{u})=d\lambda^+(\rho,{U}).\]
	\end{proposition}
	\begin{proof}
		Follows directly from Lemma \ref{lem:deral}
	\end{proof}
	
	\begin{remark}\label{remark:ana}
		Suppose $\rho\in\sHom(\Gamma,\Go,\sP_0^\pm)\subset\sHom(\Gamma,\G,\sP^\pm)$. Hence, for $t$ small enough $\rho_t\in\sHom(\Gamma,\G,\sP^\pm)$. We use Remark \ref{rem:nonparano} and Theorem \ref{thm:BCLS} to obtain that the limit maps of $\rho_t$ vary analytically along the variable $t$ (for more details please see Theorem 6.1 of \cite{BCLS}, Theorem 3.8 of \cite{HPS} and Theorem 5.18 of \cite{Shub}). 
		
		Suppose $f^\pm_t$ be the functions whose Liv\v sic cohomology classes are proto-Labourie--Margulis invariants and which are obtained via the construction of the collection of norms on $\R^{2n}$, indexed by $p\in\cflow$ and $\rho_t$ (see Remark \ref{rem:norm}. We recall that $f^\pm_t$ are H\"older continuous in the variable $p$ and vary analytically over a neigborhood of $\rho$ inside the representation variety i.e. for some collection of H\" older continuous funtions $\{h_n^\pm\}_{n=0}^\infty$ over $\flow$ we have
		\[f_t^\pm=\sum_{n=0}^\infty t^nh_n^\pm.\]
		Henceforth, to simplify our notations we denote $h_1^\pm$ by $h^\pm$ and $\sum_{n=2}^\infty t^{n-2}h_n^\pm$ by $h_t^\pm$. We note that $h_0^\pm=f_0^\pm$ and $h^\pm_t$ is analytic in the variable $t$. 
	\end{remark}
	\begin{proposition}\label{prop:lminv}
		Suppose $\{\rho_t\}_{t\in(-2,2)}\subset\sHom(\Gamma,\G,\sP^\pm)$ be an analytic one parameter family with $\rho=\rho_0\in\sHom(\Gamma,\Go,\sP_0^\pm)$, ${U}$ is the corresponding tangent vector to $\sHom(\Gamma,\G,\sP^\pm)$ at $\rho$ and ${u}={U}e_{2n}$. Then the derivative of the left proto-Labourie--Margulis invariants of $\rho_t$ at $t=0$ is the Labourie--Margulis invariant of $(\rho,u)\in\sHom(\Gamma,\Ga)$, i.e.
		\[[f_{\rho,u}] = \left[\left.\frac{\partial}{\partial t}\right|_{t=0}f^+_t\right] .\]
	\end{proposition}
	\begin{proof}
		Suppose $h^+,h^+_t:\flow\rightarrow\R$ be as mentioned in Remark \ref{remark:ana}. Hence $f^+_t=f^+_0+th^++t^2h^+_t$. It follows that for all flow invariant probability measures $\mu_\gamma$ supported on the closed orbits corresponding to infinite order elements $\gamma\in\Gamma$,
		\[\int f^+_t d\mu_\gamma=\int (f^+_0+t h^++t^2h^+_t)d\mu_\gamma.\]
		We use Proposition \ref{prop:derlminv} and deduce that
		\[{\log\lambda^+_t(\gamma)-\log\lambda^+_0(\gamma)}={l(\gamma)t}\int h^+d\mu_\gamma +{l(\gamma)}t^2\int h^+_t d\mu_\gamma.\]
		We recall that $\flow$ is compact, the functions $\{h^+_t\}_{t\in[-1,1]}$ vary analytically in the variable $t$ and $[-1,1]$ is a compact set. Hence, for $t\in[-1,1]$,
		\[\int h^+_t d\mu_\gamma\leqslant\max_{t\in[-1,1]}\max_{p\in\flow}|h^+_t(p)|=K\in\R.\]
		It follows that $\lim_{t\to0}t\int h^+_t d\mu_\gamma=0$. Also, we use Remark \ref{rem:1} to obtain
		\begin{align*}
			\int h^+d\mu_\gamma= \frac{d\lambda^+((\rho,{U})(\gamma))}{l(\gamma)\lambda^+_0(\gamma)}=\frac{\alpha((\rho,u)(\gamma))}{l(\gamma)}.
		\end{align*}
		Finally, using Liv\v sic's theorem \cite{Liv} we conclude our result.
	\end{proof}
	\begin{remark}
		Similarly, we can show that the derivative of the right proto-Labourie--Margulis invariants of $\rho_t$ at $t=0$ is the negative of the Labourie--Margulis invariant of $(\rho,u)$. 
		
		Also, we recall that the negative of the Labourie--Margulis invariant of $(\rho,u)$ is the Labourie--Margulis invariant of $(\rho,-u)$.
	\end{remark}
	
	\subsection{Margulis spacetimes and quotients of $\bH$}
	
	In this subsection we relate elements of $\sHom(\Gamma,\Ga)$ with deformations in $\sHom(\Gamma,\G)$ of elements in $\sHom(\Gamma,\Go)$. Moreover, we use this to relate proper affine actions on $\R^{n,n-1}$ with proper actions on $\bH$.
	
	\begin{lemma}
		Suppose $\rho\in\sHom(\Gamma,\Go,\sP_0^\pm)$ and 
		\[E:\sZ^1_{\sad\circ\rho}(\Gamma,\fg)\to\sZ^1_{\rho}(\Gamma,\R^{2n-1})\]
		is the map which sends ${U}\in\sZ^1_{\sad\circ\rho}(\Gamma,\fg)$ to $u=Ue_{2n}$. Then $E$ is surjective and the kernel of $E$ is $\sZ^1_{\sad\circ\rho}(\Gamma,\fg_0)$.
	\end{lemma}
	\begin{proof}
		Suppose $u\in\sZ^1_{\rho}(\Gamma,\R^{2n-1})$. We consider $U:\Gamma\to\fg$ such that,
		\[U(\gamma):=\begin{bmatrix}
			0& u(\gamma)\\
			u(\gamma)^tI_{n,n-1} & 0\end{bmatrix}.\]
		We observe that $U\in\sZ^1_{\sad\circ\rho}(\Gamma,\fg)$. As $Ue_{2n}=u$, we deduce $E$ is surjective.
		
		Suppose $U\in\sZ^1_{\sad\circ\rho}(\Gamma,\fg)$ is such that $Ue_{2n}=0$. As $U(\gamma)\in\fg$, we compute and conclude that $U(\gamma)\in\fg_0$ and our result follows.
	\end{proof}
	
	\begin{lemma}\label{lem:affano}
		Let $\{\rho_t\}_{t\in(-1,1)}\subset\sHom(\Gamma,\G,\sP^\pm)$ be an analytic one parameter family with $\rho=\rho_0\in\sHom(\Gamma,\Go,\sP_0^\pm)$, ${U}$ be the tangent vector to $\{\rho_t\}_{t\in(-1,1)}$ at $\rho$ and ${u}={U}e_{2n}$. Suppose the Labourie--Margulis invariant of $(\rho,u)$ is non-vanishing. Then there exists $\epsilon\in(0,1)$ such that for all $t$ with $|t|\in(0,\epsilon)$, the proto-Labourie--Margulis invariants of $\rho_t$ are also non-vanishing.
	\end{lemma}
	\begin{proof}
		We use Remark \ref{remark:ana} to deduce the existence of H\"older continuous functions $h^\pm, h^\pm_t:\flow\rightarrow\R$ such that $h^\pm_t$ vary analytically over $t$ in a neighborhood of zero (i.e. $|t|<\epsilon$ for some $\epsilon\in(0,1)$) with 
		\[f^\pm_t=f^\pm_0+th^\pm+t^2h^\pm_t\]
		and $[f^\pm_t]$ are respectively the proto-Labourie--Margulis invariants of $\rho_t$. We use Proposition \ref{prop:lminv} to deduce that $h^\pm$ are Liv\v sic cohomologous to non-vanishing functions. Without loss of generality suppose $h^+$ is Liv\v sic cohomologous to a function $h>0$. We define 
		\[f_t\defeq th+t^2h^+_t.\]
		Moreover, let us consider $c>\max\{|h^+_s(p)|\mid p\in\flow, s\in[-\epsilon,\epsilon]\}$. Then for $|t|\in(0,\epsilon)$ with $c |t| <\min\{h(p)\mid p\in\flow\}$ we have 
		\[h+th^+_t\geqslant h-|th^+_t|> h-c|t|>0.\]
		It follows that for all $t$ with $|t|\in(0,\epsilon)$, the functions $f_t$ are non-vanishing. We recall that $h$ is Liv\v sic cohomologous to $h^+$ and use Proposition \ref{prop:derlminv} to deduce that
		\[\int (f^+_t-f_t) d\mu_\gamma=\int f_0 d\mu_\gamma + t\int (h^+-h) d\mu_\gamma=\frac{\log1}{l(\gamma)}+0=0.\]
		Hence, using Liv\v sic's Theorem \cite{Liv} we obtain that $[f^+_t]=[f_t]$ and our result follows.
	\end{proof}
	
	\begin{theorem}\label{thm:affano}
		Let $\{\rho_t\}_{t\in(-1,1)}$ be an analytic one parameter family of representations of $\Gamma$ in $\G$ with $\rho_0(\Gamma)\subset\Go$. Let ${U}$ be the tangent vector to $\{\rho_t\}_{t\in(-1,1)}$ at $\rho=\rho_0$ and ${u}={U}e_{2n}$. Suppose  $(\rho,u)\in\sHom(\Gamma,\Ga,\sP^\pm_a)$. Then there exists $\epsilon>0$ such that for all $t$ with $|t|\in(0,\epsilon)$, $\rho_t$ is Anosov in $\mathsf{SL}(2n,\R)$ with respect to the stabilizer of an $n$-dimensional subspace.
	\end{theorem}
	
	\begin{proof}
		As $(\rho,u)\in\sHom(\Gamma,\Ga,\sP^\pm_a)$, we obtain that $\rho\in\sHom(\Gamma,\Go,\sP^\pm_0)$. We use Proposition \ref{prop:defano} to obtain an $\epsilon>0$ such that for all $|t|<2\epsilon$, $\rho_t\in\sHom(\Gamma,\G,\sP^\pm)$. Finally, we use Lemma \ref{lem:affano}, Proposition \ref{prop:meano1} and Remark \ref{rem:unorien} to conclude our result.
	\end{proof}
	
	\begin{corollary}\label{cor:pexist}
		Let $\{\rho_t\}_{t\in(-1,1)}$ be an analytic one parameter family of representations of $\Gamma$ in $\G$ with $\rho=\rho_0\in\sHom(\Gamma,\Go,\sP_0^\pm)$, ${U}$ be the tangent vector to $\{\rho_t\}_{t\in(-1,1)}$ at $\rho$ and ${u}={U}e_{2n}$. Suppose $(\rho,u)$ is a Margulis spacetime. Then there exists $\epsilon>0$ such that for all $t$ with $|t|\in(0,\epsilon)$, $\rho_t(\Gamma)$ acts properly on $\bH$.
	\end{corollary}
	
	\begin{proof}
		As $(\rho,u)$ is a Margulis spacetime, we recall that $(\rho,u)(\Gamma)$ acts properly on $\R^{2n-1}$. We use Theorem \ref{thm:GT} to obtain that $(\rho,u)\in\sHom(\Gamma,\Ga,\sP^\pm_a)$. Finally, we use Theorems \ref{thm:affano} and \ref{thm:p} to conclude our result.
	\end{proof}
	
	\begin{corollary}
		Suppose $n$ is even. Then there exists a non-abelian free subgroup with finitely many generators inside $\G$ which act properly on $\bH$.
	\end{corollary}
	\begin{proof}
		The result follows from Theorem B of \cite{AMS} and Corollary \ref{cor:pexist}.
	\end{proof}

	\section*{Appendix}
	In this section we introduce affine crossratios corresponding to Margulis invariants and crossratios corresponding to the eigenvalues whose derivative give rise to Margulis invariants. Moreover, we relate these affine crossratios (resp. crossratios) with a limiting result corresponding the Margulis invariants (resp. eigenvalues).
	\subsection{Affine Crossratios and Margulis Invariants}\label{sec:acr}
	In this subsection we define affine crossratios. 
	
	Suppose $\{V_i\}_{i=1}^{4}$ are four null vector subspaces of $\R^{2n-1}$ which are mutually transverse to each other and $\{A_i\}_{i=1}^{4}$ are four affine subspaces in $\R^{2n-1}$ such that $V_i$ is respectively parallel to $A_i$.  Moreover, for $i\neq j$ suppose $x_{i,j}$ be a point in $A_i\cap A_j$ and suppose $v_{i,j}\defeq \nu(V_i^\perp,V_j^\perp)$. 
	\begin{lemma}\label{lem:signu}
		Suppose $V_i,V_j,V_k$ are three null vector subspaces of $\R^{2n-1}$ which are mutually transverse to each other. Then $v_{i,j}=(-1)^{n-1}v_{j,i}$ and
		\[\langle v_{i,j}\mid v_{i,k}\rangle=1=\langle v_{i,j}\mid v_{k,j}\rangle.\]
	\end{lemma}
	\begin{proof}
		As $\dim(V_i)=n$, $\dim(V_i^\perp)=(n-1)$ and $\langle v_{i,*}\mid v_{i,*}\rangle=1$, we obtain that $(v_{i,j}-av_{i,k})\in V_i^\perp$ for some non-zero constant $a$. It follows that 
		\[a=\langle v_{i,j}\mid v_{i,k}\rangle=a^{-1}.\]
		Hence, $a^2=1$. As $\langle v_{i,j}\mid v_{i,j}\rangle=1$, using continuity we conclude that $a=1$.
		
		We note that $-I$ and $I$ lie in the same connected component of the orthogonal group $\mathsf{O}(n)$ if and only if $n$ is even. Hence, $I_{n,n-1}\in\mathsf{SO}_0(n,n-1)$ for $n$ odd and $-I_{n,n-1}\in\mathsf{SO}_0(n,n-1)$ for $n$ even. It follows that for $n$ odd, $\nu(W_-,W_+)=\nu([I_{n,n-1}])=\nu(W_+,W_-)=(-1)^{n-1}\nu(W_+,W_-)$ and for $n$ even,
		$\nu(W_-,W_+)=\nu([-I_{n,n-1}])=-\nu(W_+,W_-)=(-1)^{n-1}\nu(W_+,W_-)$. Our result follows.
	\end{proof}
	We define the \emph{affine crossratio} as
	\[\beta_{1,2,3,4}=\beta(A_1,A_2,A_3,A_4)\defeq \langle x_{1,3}-x_{2,4}\mid v_{1,4}-v_{2,3} \rangle.\]
	
	In particular, for $n=2$, null subspaces of $\R^{2,1}$ are planes which are tangent to the light-cone. Now given four affine subspaces of $\R^3$ which are mutually transverse to each other and whose underlying vector spaces are null subspaces, we obtain that their mutual intersections give us four affine lines. Then $x_{1,3}$ (resp. $x_{2,4}$) is a point on the line of intersection between the first (resp. second) and the third (resp. fourth) affine subspace and $v_{1,4}$ (resp. $v_{2,3}$) are vectors which are unit with respect to the bilinear form $\langle\mid\rangle$, parallel with the intersection between the first (resp. second) and the fourth (resp. third) affine subspace and whose directions are consistent with the choice made in Remark \ref{rem:nuor}. 
	
	The above definition is well defined since using Lemma \ref{lem:signu} it follows that for all $a,b\in\R$, $\langle av_{1,3}-bv_{2,4}\mid v_{1,4}-v_{2,3} \rangle=0$. We also observe that the following equality holds for any other points $x_{3,1}\in A_3\cap A_1$ and $x_{4,2}\in A_4\cap A_2$:
	\[ \langle x_{1,3}-x_{2,4}\mid v_{1,4}-v_{2,3} \rangle= \langle x_{3,1}-x_{4,2}\mid v_{1,4}-v_{2,3} \rangle.\]
	Now for $i\neq j$ we consider the following decomposition:
	\[\R^{2n-1}=V_i^\perp\oplus V_j^\perp\oplus (V_i\cap V_j).\]  
	Let $x_i^j$ be the projection of $x_{i,j}$ on $V_i^\perp$ with respect to this decomposition. We observe that as $x_{i,j}$ varies along $A_i\cap A_j$ the projection $x_i^j$ stays fixed. Moreover, $x_i^j+V_j=A_j$. Using these observations we obtain:
	\begin{align*}
		&\beta_{1,2,3,4}=\langle x_{1,3}-x_{2,4}\mid v_{1,4}-v_{2,3} \rangle=\langle x_3^1+ x_1^3-x_2^4-x_4^2\mid v_{1,4}-v_{2,3}\rangle\\
		&=\langle x_3^1\mid v_{1,4}\rangle-\langle x_1^3\mid v_{2,3}\rangle -\langle x_2^4\mid v_{1,4}\rangle +\langle x_4^2\mid v_{2,3}\rangle\\
		&=\langle x_3^1\mid v_{1,4}-v_{1,3}\rangle-\langle x_1^3\mid v_{2,3}-v_{1,3}\rangle-\langle x_2^4\mid v_{1,4}-v_{2,4}\rangle+\langle x_4^2\mid v_{2,3}-v_{2,4}\rangle.
	\end{align*}
	Hence for any $x_i\in A_i$ we obtain the following identity:
	\begin{align}\label{identity:altmarg}
		\beta_{1,2,3,4}=\langle x_1\mid v_{1,4}-v_{1,3}\rangle&+\langle x_2\mid v_{2,3}-v_{2,4}\rangle\\
		&+\langle x_3\mid v_{1,3}-v_{2,3}\rangle+\langle x_4\mid v_{2,4}-v_{1,4}\rangle.\notag
	\end{align}
	\begin{proposition}\label{prop:cr}
		Let $\beta$ be defined as above. Then for any five affine null spaces $A_*,\{A_i\}_{i=1}^{4}$ which are mutually transverse to each other and for any $(g,u)\in\Ga$ the following identities hold:
		\begin{enumerate}
			\item $\beta((g,u)A_1,(g,u)A_2,(g,u)A_3,(g,u)A_4)=\beta(A_1,A_2,A_3,A_4)$,
			\item \label{identity:afcros} $\beta_{1,2,3,4}=\beta_{2,1,4,3}=(-1)^n\beta_{3,4,1,2}=(-1)^n\beta_{4,3,2,1}$,
			\item $\beta_{1,2,3,4}+\beta_{1,2,4,3}=0$,
			\item $\beta_{1,*,3,4}+\beta_{*,2,3,4}=\beta_{1,2,3,4}$.
		\end{enumerate}
		Moreover, for $n$ even, $\beta_{1,2,3,4}+\beta_{1,3,4,2}+\beta_{1,4,2,3}=0$.
	\end{proposition}
	\begin{proof}
		We use the definition of $\beta$ to deduce that for all $(g,u)\in\Ga$,
		\[\beta((g,u)A_1,(g,u)A_2,(g,u)A_3,(g,u)A_4)=\beta(A_1,A_2,A_3,A_4).\]
		We recall that $\nu(V,W)=(-1)^{n-1}\nu(W,V)$. Moreover, exploiting the symmetries in the definition of $\beta$ we obtain the identity \ref{identity:afcros}. 
		Now interchanging $A_3$ and $A_4$ in the identity \ref{identity:altmarg} and adding them up we obtain that $\beta_{1,2,3,4}+\beta_{1,2,4,3}=0$. 
		Suppose $A_*$ is another affine null space which is mutually transverse with the other null spaces $\{A_i\}_{i=1}^4$. We observe that
		\begin{align*}
			&\langle x_1\mid v_{1,4}-v_{1,3}\rangle+\langle x_*\mid v_{*,3}-v_{*,4}\rangle+\langle x_3\mid v_{1,3}-v_{*,3}\rangle+\langle x_4\mid v_{*,4}-v_{1,4}\rangle\\
			&+\langle x_*\mid v_{*,4}-v_{*,3}\rangle+\langle x_2\mid v_{2,3}-v_{2,4}\rangle+\langle x_3\mid v_{*,3}-v_{2,3}\rangle+\langle x_4\mid v_{2,4}-v_{*,4}\rangle\\
			&=\langle x_1\mid v_{1,4}-v_{1,3}\rangle+\langle x_2\mid v_{2,3}-v_{2,4}\rangle+\langle x_3\mid v_{1,3}-v_{2,3}\rangle+\langle x_4\mid v_{2,4}-v_{1,4}\rangle.
		\end{align*}
		Therefore, we conclude that $\beta_{1,*,3,4}+\beta_{*,2,3,4}=\beta_{1,2,3,4}$.
		
		Finally, for even $n$, we cyclically permute $A_2,A_3,A_4$ in the identity \ref{identity:altmarg} and add them up to conclude our result.
	\end{proof}
	
	\begin{proposition}\label{prop:alphabeta}
		Suppose $(g,u)\in\Ga$ be such that its action on the space of affine null subspaces has an attracting (resp. repelling) fixed point $A_+$ (resp. $A_-$) and $A_\pm$ are transverse to each other. Then for any affine null space $A$ which is transverse to both $A_\pm$ the following holds:
		\begin{enumerate}
			\item $\beta(A_-,A_+,(g,u) A, A)=2\alpha(g,u)$ when $n$ is even,
			\item $\beta(A_-,A_+,(g,u) A, A)=0$ when $n$ is odd.
		\end{enumerate}
	\end{proposition}
	\begin{proof}
		Suppose $h=(g,u)$ and $x_\pm,x,x_h$ are any four points respectively in $A_\pm$, $A$ and $hA$. We observe that $\langle x_\pm\mid v_{A_\pm,h A}\rangle=\langle g^{-1}x_\pm\mid v_{A_\pm,A}\rangle$. Hence,
		\[\langle x_\pm\mid v_{A_\pm,A}-v_{A_\pm,h A}\rangle=\langle x_\pm-g^{-1}x_\pm\mid v_{A_\pm,A}\rangle=\langle x_\pm-h^{-1}x_\pm-g^{-1}u\mid v_{A_\pm,A}\rangle.\]
		As $x_h=hx^\prime$ for some $x^\prime\in A$, we obtain $g^{-1}x_h=g^{-1}u+ x^\prime$. Suppose $V$ is the underlying vector subspace of $A$. We observe that $x^\prime-x\in V$ and $(v_{A_-,A}-v_{A_+,A})\in V^\perp$. It follows that
		\[\langle x_h\mid v_{A_-,h A}-v_{A_+,h A}\rangle=\langle g^{-1}u+x\mid v_{A_-,A}-v_{A_+,A}\rangle.\]
		We use the identity \ref{identity:altmarg} and deduce that
		\[\beta(A_-,A_+,h A, A)=\langle x_--h^{-1}x_-\mid v_{A_-,A}\rangle + \langle h^{-1}x_+-x_+\mid v_{A_+,A}\rangle.\]
		Suppose $V_\pm$ are the vector spaces which are respectively parallel to $A_\pm$. We recall that $h$ fixes $A_\pm$ and hence $(x_\pm-h^{-1}x_\pm)\in V_\pm$. On the other hand $(v_{A_\pm,A}-v_{A_\pm,A_\mp})\in V_\pm^\perp$ and therefore we deduce that
		\[\beta(A_-,A_+,(g,u) A, A)=\langle u\mid v_{A_-,A_+}-v_{A_+,A_-}\rangle.\]
		As $v_{i,j}=(-1)^{n-1}v_{j,i}$, our result follows.
	\end{proof}
	\begin{remark}
		Suppose $(\rho,u)\in\sHom(\Gamma,\Ga)$ is such that $\rho\in\sHom(\Gamma,\Go,\sP_0^\pm)$. Hence, $(\rho,u)$ admits limit maps $\xi^\pm:\bdry\rightarrow\Ga/\sP_a^\pm$ which satisfy the first two properties of being an affine Anosov representation (for more details please see Proposition 5.3 of \cite{GT}). In general, it only fails to satisfy the third property. Hence, for all infinite order elements $\gamma\in\Gamma$ we obtain that the action of $({\rho,u})(\gamma)$ on $\cX_a$ has an attracting fixed point and a repelling fixed point. We abuse notation and let $\xi_{\rho,u}(\gamma^+)$ (resp. $\xi_{\rho,u}(\gamma^-)$) denote the attracting (resp. repelling) fixed point. Henceforth, we fix the representation $(\rho,u)$ and omit the subscripts $(\rho,u)$ from the notation of the Margulis invariants and the affine crossratios. Also, when there is no confusion of notation, for $a,b,c,d\in\bdry$ all distinct, we denote $\beta(\xi(a), \xi(b), \xi(c), \xi(d))$ by $\beta(a,b,c,d)$.
	\end{remark}
	\begin{proposition}\label{prop:lim1}
		Suppose $n$ is even, $\rho\in\sHom(\Gamma,\Go,\sP_0^\pm)$ and $(\rho,u)\in\sHom(\Gamma,\Ga)$. Also, suppose $\gamma,\eta\in\Gamma$ are two infinite order elements such that the four points $\gamma^\pm,\eta^\pm\in\bdry$ are distinct and the sequence $\{\gamma^m\eta^k\}_{m\in\mathbb{N}}\subset\Gamma$ contains a subsequence $\{\gamma^{n_i}\eta^k\}_{i\in\mathbb{N}}$ consisting only of infinite order elements. Then the following identity holds:
		\begin{align*}
			{\beta(\eta^-, \gamma^-,\gamma^+,\eta^k\gamma^+)+\beta(\eta^+,\gamma^+,\gamma^-,\eta^{-k}\gamma^-)}=2\lim_{i\to\infty}[\alpha(\gamma^{n_i}\eta^k)&-\alpha(\gamma^{n_i})]\\&-2\alpha(\eta^k).
		\end{align*}
	\end{proposition}
	\begin{proof}
		Suppose $\xi$ is the affine limit map as mentioned in the previous remark and $A,B,C$ are affine null spaces such that $A$ is transverse to both $\xi((\gamma^{n_i}\eta^k)^\pm)$, $B$ is transverse to both $\xi(\gamma^\pm)$ and $C$ is transverse to both $\xi(\eta^\pm)$. We use Proposition \ref{prop:alphabeta} and obtain the following three identities:
		\begin{align*}
			2\alpha(\gamma^{n_i}\eta^k)&=\beta(\xi((\gamma^{n_i}\eta^k)^-),\xi((\gamma^{n_i}\eta^k)^+),(\rho,u)(\gamma^{n_i}\eta^k) A, A),\\
			2\alpha(\gamma^{n_i})&=\beta(\xi(\gamma^-),\xi(\gamma^+),(\rho,u)(\gamma^{n_i}) B, B),\\
			2\alpha(\eta^k)&=\beta(\xi(\eta^-),\xi(\eta^+),(\rho,u)(\eta^k) C, C).
		\end{align*}
		We observe that $\lim_{i\to\infty}(\gamma^{n_i}\eta^k)^+=\gamma^+$ and $\lim_{i\to\infty}(\gamma^{n_i}\eta^k)^-=\eta^{-k}\gamma^-$. Also, $\eta^\pm\neq\gamma^\pm$. It follows that $\eta^-\neq \lim_{i\to\infty}(\gamma^{n_i}\eta^k)^\pm$. Hence, we can choose $A=B=\xi(\eta^-)$, $C=\xi(\gamma^-)$. We use Proposition \ref{prop:cr} (4) to obtain:
		\begin{align*}
			2\alpha(\gamma^{n_i}\eta^k)   &=\beta(\gamma^+,(\gamma^{n_i}\eta^k)^+,\gamma^{n_i}\eta^-, \eta^-)+\beta((\gamma^{n_i}\eta^k)^-,\gamma^+,\gamma^{n_i}\eta^-, \eta^-),\\
			2\alpha(\gamma^{n_i})
			&=\beta((\gamma^{n_i}\eta^k)^-,\gamma^+,\gamma^{n_i}\eta^-, \eta^-)+\beta(\gamma^-,(\gamma^{n_i}\eta^k)^-,\gamma^{n_i}\eta^-, \eta^-).
		\end{align*}
		Also, using Proposition \ref{prop:cr} (1) we deduce that
		\begin{align*}
			\beta(\gamma^+,(\gamma^{n_i}\eta^k)^+,\gamma^{n_i}\eta^-, \eta^-)&=\beta(\gamma^+,(\eta^k\gamma^{n_i})^+,\eta^-, \gamma^{-n_i}\eta^-),\\ 2\alpha(\eta^k)&=\beta(\eta^-,\eta^+,\gamma^-,\eta^{-k}\gamma^-).
		\end{align*}
		Hence, taking the limit and then using Proposition \ref{prop:cr} (2) and (3) we obtain:
		\begin{align*}
			2\lim_{i\to\infty}[\alpha(\gamma^{n_i}\eta^k)-\alpha(\gamma^{n_i})]
			&=\beta(\gamma^+,\eta^k\gamma^+,\eta^-, \gamma^-)-\beta(\gamma^-,\eta^{-k}\gamma^-,\gamma^+,\eta^-)\\
			&=\beta(\eta^-, \gamma^-,\gamma^+,\eta^k\gamma^+)+\beta(\eta^-,\gamma^+,\gamma^-,\eta^{-k}\gamma^-).
		\end{align*}
		Now we use Proposition \ref{prop:cr} (4) to deduce that
		\[\beta(\eta^-,\gamma^+,\gamma^-,\eta^{-k}\gamma^-)=\beta(\eta^-,\eta^+,\gamma^-, \eta^{-k}\gamma^-)+\beta(\eta^+,\gamma^+,\gamma^-,\eta^{-k}\gamma^-).\]
		Finally, our result follows from combining the last three identities.
	\end{proof}
	\begin{proposition}\label{prop:lim2}
		Suppose $n$ is even, $\rho\in\sHom(\Gamma,\Go,\sP_0^\pm)$ and $(\rho,u)\in\sHom(\Gamma,\Ga)$. Suppose $\gamma,\eta\in\Gamma$ are two infinite order elements such that the four points $\gamma^\pm,\eta^\pm\in\bdry$ are distinct and the sequence $\{\gamma^m\eta^m\}_{m\in\mathbb{N}}\subset\Gamma$ contains a subsequence $\{\gamma^{n_i}\eta^{n_i}\}_{i\in\mathbb{N}}$ consisting only of infinite order elements. Then the following identity holds:
		\[\lim_{i\to\infty}(\alpha(\gamma^{n_i}\eta^{n_i})-\alpha(\gamma^{n_i})-\alpha(\eta^{n_i}))=\beta(\eta^-, \gamma^-,\gamma^+,\eta^+).\]
	\end{proposition}
	\begin{proof}
		Suppose $\xi$ is the corresponding limit map and suppose $\{A_i,B_i,C_i\}_{i\in\mathbb{N}}$ is a collection of affine null spaces such that $A_i$ is transverse to $\xi((\gamma^{n_i}\eta^{n_i})^\pm)$, $B_i$ is transverse to $\xi(\gamma^\pm)$ and $C_i$ is transverse to $\xi(\eta^\pm)$. We use Proposition \ref{prop:alphabeta} to obtain the following three identities:
		\begin{align*}
			2\alpha(\gamma^{n_i}\eta^{n_i})=&\beta(\xi((\gamma^{n_i}\eta^{n_i})^-),\xi((\gamma^{n_i}\eta^{n_i})^+),(\rho,u)(\gamma^{n_i}\eta^{n_i}) A_i, A_i),\\
			2\alpha(\gamma^{n_i})=&\beta(\xi(\gamma^-),\xi(\gamma^+),(\rho,u)(\gamma^{n_i}) B_i, B_i),\\
			2\alpha(\eta^{n_i})=&\beta(\xi(\eta^-),\xi(\eta^+),(\rho,u)(\eta^{n_i}) C_i, C_i).
		\end{align*}
		Suppose $D_i\defeq(\rho,u)(\eta^{n_i})A_i$. We use Proposition \ref{prop:cr} (1) and (4) to deduce
		\begin{align*}
			\beta(\xi((\gamma^{n_i}\eta^{n_i})^-)&,\xi((\gamma^{n_i}\eta^{n_i})^+),(\rho,u)(\gamma^{n_i}) D_i, (\rho,u)(\eta^{-n_i})D_i)\\
			=&\beta(\xi((\gamma^{n_i}\eta^{n_i})^-),\xi((\gamma^{n_i}\eta^{n_i})^+),(\rho,u)(\gamma^{n_i}) D_i, D_i)\\
			&+\beta(\xi((\gamma^{n_i}\eta^{n_i})^-),\xi((\gamma^{n_i}\eta^{n_i})^+), D_i, (\rho,u)(\eta^{-n_i})D_i)\\
			=&\beta(\xi((\gamma^{n_i}\eta^{n_i})^-),\xi((\gamma^{n_i}\eta^{n_i})^+),(\rho,u)(\gamma^{n_i}) D_i, D_i)\\
			&+\beta(\xi((\eta^{n_i}\gamma^{n_i})^-),\xi((\eta^{n_i}\gamma^{n_i})^+), (\rho,u)(\eta^{n_i}) D_i, D_i).
		\end{align*}
		Moreover, by applying Proposition \ref{prop:cr} (4) twice we deduce that 
		\begin{align*}
			\beta(\xi((\gamma^{n_i}\eta^{n_i})^-)&,\xi((\gamma^{n_i}\eta^{n_i})^+),(\rho,u)(\gamma^{n_i}) D_i, D_i)\\
			=&\beta(\xi((\gamma^{n_i}\eta^{n_i})^-),\xi(\gamma^-),(\rho,u)(\gamma^{n_i}) D_i, D_i)\\
			&+\beta(\xi(\gamma^+),\xi((\gamma^{n_i}\eta^{n_i})^+),(\rho,u)(\gamma^{n_i}) D_i, D_i)\\
			&+\beta(\xi(\gamma^-),\xi(\gamma^+),(\rho,u)(\gamma^{n_i}) D_i, D_i).
		\end{align*}
		Similarly, we also have
		\begin{align*}
			\beta(\xi((\eta^{n_i}\gamma^{n_i})^-)&,\xi((\eta^{n_i}\gamma^{n_i})^+),(\rho,u)(\eta^{n_i}) D_i, D_i)\\
			=&\beta(\xi((\eta^{n_i}\gamma^{n_i})^-),\xi(\eta^-),(\rho,u)(\eta^{n_i}) D_i, D_i)\\
			&+\beta(\xi(\eta^+),\xi((\eta^{n_i}\gamma^{n_i})^+),(\rho,u)(\eta^{n_i}) D_i, D_i)\\
			&+\beta(\xi(\eta^-),\xi(\eta^+),(\rho,u)(\eta^{n_i}) D_i, D_i).
		\end{align*}
		We observe that $\lim_{i\to\infty}(\gamma^{n_i}\eta^{n_i})^+=\gamma^+$, $\lim_{i\to\infty}(\gamma^{n_i}\eta^{n_i})^-=\eta^-$ and also $\lim_{i\to\infty}(\eta^{n_i}\gamma^{n_i})^+=\eta^+$, $\lim_{i\to\infty}(\eta^{n_i}\gamma^{n_i})^-=\gamma^-$. We recall that the four points $\eta^\pm,\gamma^\pm$ are distinct. Let $x\in\bdry$ be such that it is distinct from all the following four points: $\gamma^\pm,\eta^\pm$. Hence without loss of generality we can choose $B_i=C_i=D_i=\xi(x)$ for all $i\in\mathbb{N}$.
		It follows that
		\begin{align*}
			2(\alpha(\gamma^{n_i}\eta^{n_i})&-\alpha(\gamma^{n_i})-\alpha(\eta^{n_i}))\\
			=&\beta((\gamma^{n_i}\eta^{n_i})^-,\gamma^-,\gamma^{n_i}x, x)+\beta(\gamma^+,(\gamma^{n_i}\eta^{n_i})^+,\gamma^{n_i}x, x)\\
			&+\beta((\eta^{n_i}\gamma^{n_i})^-,\eta^-,\eta^{n_i}x, x)+\beta(\eta^+,(\eta^{n_i}\gamma^{n_i})^+,\eta^{n_i}x, x)\\
			=&\beta((\gamma^{n_i}\eta^{n_i})^-,\gamma^-,\gamma^{n_i}x, x)+\beta(\gamma^+,(\eta^{n_i}\gamma^{n_i})^+,x, \gamma^{-n_i}x)\\
			&+\beta((\eta^{n_i}\gamma^{n_i})^-,\eta^-,\eta^{n_i}x, x)+\beta(\eta^+,(\gamma^{n_i}\eta^{n_i})^+, x , \eta^{-n_i} x).
		\end{align*}
		Finally, we observe that
		\begin{align*}
			\lim_{i\to\infty}\beta((\gamma^{n_i}\eta^{n_i})^-,\gamma^-,\gamma^{n_i} x, x)=& \beta(\eta^-,\gamma^-,\gamma^+, x),\\
			\lim_{i\to\infty}\beta(\gamma^+,(\eta^{n_i}\gamma^{n_i})^+, x, \gamma^{-n_i} x)=& \beta(\gamma^+,\eta^+, x, \gamma^-)=\beta(x, \gamma^-, \gamma^+, \eta^+),\\
			\lim_{i\to\infty}\beta((\eta^{n_i}\gamma^{n_i})^-,\eta^-,\eta^{n_i}x, x)=& \beta(\gamma^-,\eta^-,\eta^+, x)=\beta(\eta^-, \gamma^-, x, \eta^+),\\
			\lim_{i\to\infty}\beta(\eta^+,(\gamma^{n_i}\eta^{n_i})^+, x , \eta^{-n_i} x)=& \beta(\eta^+,\gamma^+, x , \eta^-)=\beta(\eta^-,x,\gamma^+,\eta^+),
		\end{align*}
		and conclude our result using Proposition \ref{prop:cr} (4).
	\end{proof}

	\subsection{Crossratios and Eigenvalues}\label{sec:cr}
	In this subsection we define, for the linear case, appropriate counterparts of the affine crossratios. Affine crossratios can be seen as infinitesimal versions of these crossratios. 
	
	Let $\{W_i\}_{i=1}^4$ be four $(n-1)$-dimensional isotropic subspaces in $\R^{n,n}$ such that their orthogonal spaces are mutually transverse to each other. We recall that for $i\neq j$, $W_i^\perp\cap W_j^\perp$ contain exactly two isotropic lines. We use Lemmas \ref{lem:stiso} and \ref{lem:choice} to  choose $v_{i,j}^\pm$ arbitrarily from one of these two lines such that $\R  v_{i,j}^+\oplus W_i$ (resp. $\R  v_{i,j}^-\oplus W_i$) lies in the orbit of $\mathsf{cspan}([I_n,I_n]^t)$ (resp. $J\mathsf{cspan}([I_n,I_n]^t)$) under the action of $\mathsf{SO}_0(n,n)$.  
	
	We define the following \emph{crossratio}: 
	\[\theta_{1,2,3,4}=\theta(W_1,W_2,W_3,W_4)\defeq\frac{\langle v^+_{1,3}\mid v^-_{2,3}\rangle\langle v^+_{2,4}\mid v^-_{1,4}\rangle}{\langle v^+_{2,4}\mid v^-_{2,3}\rangle\langle v^+_{1,3}\mid v^-_{1,4}\rangle}.\]
	As $v^\pm_{i,j}$ are unique upto scaling, we observe that the above expression does not depend on the choice of the vectors $v^\pm_{i,j}$ and hence is well defined.
	
	\begin{lemma}\label{lem:id}
		Let $W_*,W_i,W_j,W_k$ be four $(n-1)$ dimensional isotropic subspaces such that their orthogonal spaces are mutually transverse to each other. Then the following identity holds:
		\[\frac{\langle v^+_{*,i}\mid v^-_{*,i}\rangle \langle v^+_{*,j} \mid v^-_{*,k}\rangle}{\langle v^+_{*,j}\mid v^-_{*,i}\rangle\langle v^+_{*,i} \mid
			v^-_{*,k}\rangle}=1=\frac{\langle v^+_{i,*}\mid v^-_{i,*}\rangle \langle v^+_{j,*} \mid v^-_{k,*}\rangle}{\langle v^+_{j,*}\mid v^-_{i,*}\rangle\langle  v^+_{i,*}\mid v^-_{k,*}\rangle}.\]
		Moreover, for $n$ even, $\R v_{i,j}^\pm=\R v_{j,i}^\mp$ and for $n$ odd, $\R v_{i,j}^\pm= \R v_{j,i}^\pm$.
	\end{lemma}
	\begin{proof}
		As $\R  v_{*,y}^\pm\oplus W_* \subset W_*^\perp$ is a maximal isotropic subspace for $y=i,j,k$ and $\R  v_{*,y}^+\oplus W_*$ (resp. $\R  v_{*,y}^-\oplus W_*$)  lie in the orbit of $\mathsf{cspan}([I_k,I_k]^t)$ (resp. $J\mathsf{cspan}([I_k,I_k]^t)$) under the action of $\mathsf{SO}_0(n,n)$, we deduce that there exist non-zero constants $a^\pm,b^\pm$ such that
		\[v_{*,i}^\pm-a^\pm v_{*,j}^\pm \in W_* \ni v_{*,i}^\pm-b^\pm v_{*,k}^\pm.\]
		Hence, $\langle v^+_{*,i}\mid v^-_{*,i}\rangle=a^+\langle v^+_{*,j}\mid v^-_{*,i}\rangle=b^-\langle v^+_{*,i}\mid v^-_{*,k}\rangle=a^+b^-\langle v^+_{*,j}\mid v^-_{*,k}\rangle$ and the left hand side of the identity follows.
		
		Moreover, we observe that $I_{n,n}W_\pm=W_\mp=-I_{n-1,n+1}W_\pm$, $I_{n,n}v_\pm=v_\mp$ and $-I_{n-1,n+1}v_\pm=v_\pm$. As $I_{n,n}\in\mathsf{SO}_0(n,n)$ for $n$ even, we obtain that $\nu_\pm([I_{n,n}])=\nu_\mp([I])$. As $-I_{n-1,n+1}\in\mathsf{SO}_0(n,n)$ for $n$ odd, we obtain that $\nu_\pm([-I_{n-1,n+1}])=\nu_\pm([I])$. Suppose $g\in\mathsf{SO}_0(n,n-1)$ is such that $gW_+=W_i$ and $gW_-=W_j$. Then $\R v_{i,j}^\pm=\R\nu_\pm([g])$. Finally, we conclude by observing that, for $n$ even,
		$\R v_{i,j}^\pm=\R\nu_\pm([g])=\R\nu_\mp([gI_{n,n}])=\R v_{j,i}^\mp$ and for $n$ odd, $\R v_{i,j}^\pm=\R\nu_\pm([g])=\R\nu_\pm([-gI_{n-1,n+1}])=\R v_{j,i}^\pm$.
	\end{proof}
	\begin{proposition}\label{prop:ecr}
		Let $\{W_i\}_{i=1}^4$ and $W_*$ be five $(n-1)$ dimensional isotropic subspaces such that their orthogonal spaces are transverse to each other and let $g\in\G$. Then the following identities hold:
		\begin{enumerate}
			\item $\theta(gW_1,gW_2,gW_3,gW_4)=\theta(W_1,W_2,W_3,W_4)$,
			\item $\theta_{1,2,3,4}=\theta_{2,1,4,3}=\theta_{3,4,1,2}^{(-1)^{n}}=\theta_{4,3,2,1}^{(-1)^{n}}$,
			\item $\theta_{1,2,3,4}\theta_{1,2,4,3}=1$,
			\item $\theta_{1,*,3,4}\theta_{*,2,3,4}=\theta_{1,2,3,4}$.
		\end{enumerate}
		Moreover, for $n$ even, $\theta_{1,2,3,4}\theta_{1,3,4,2}\theta_{1,4,2,3}=1$.
	\end{proposition}
	\begin{proof}
		The first two identity follows from the definition of $\theta$ and Lemma \ref{lem:id}. Also, the third identity follows from Lemma \ref{lem:id} by taking $j=k$. 
		
		We use the definition of $\theta$ and cancel the terms appearing both in the numerator and denominator to see that
		\[\frac{\theta_{*,2,3,4}\theta_{1,*,3,4}}{\theta_{1,2,3,4}}=\frac{\langle v^+_{*,3}\mid v^-_{2,3}\rangle\langle v^+_{2,4}\mid v^-_{*,4}\rangle}{\langle v^+_{*,3}\mid v^-_{*,4}\rangle\langle v^+_{*,4}\mid v^-_{*,3}\rangle}\frac{\langle v^+_{1,3}\mid v^-_{*,3}\rangle\langle v^+_{*,4}\mid v^-_{1,4}\rangle}{\langle v^+_{1,3}\mid v^-_{2,3}\rangle\langle v^+_{2,4}\mid v^-_{1,4}\rangle}.\]
		The fourth identity follows by replacing the above formula by the following identities which are obtained by repeated application of Lemma \ref{lem:id}:
		\begin{align*}
			\langle v^+_{*,3}\mid v^-_{*,4}\rangle\langle v^+_{*,4}\mid v^-_{*,3}\rangle&=\langle v^+_{*,3}\mid v^-_{*,3}\rangle\langle v^+_{*,4}\mid v^-_{*,4}\rangle,\\
			\langle v^+_{*,3}\mid v^-_{*,3}\rangle \langle v^+_{1,3}\mid v^-_{2,3}\rangle &=\langle v^+_{1,3}\mid v^-_{*,3}\rangle\langle v^+_{*,3}\mid v^-_{2,3}\rangle,\\
			\langle v^+_{*,4}\mid v^-_{*,4}\rangle \langle v^+_{2,4}\mid v^-_{1,4}\rangle &=\langle v^+_{2,4}\mid v^-_{*,4}\rangle\langle v^+_{*,4}\mid v^-_{1,4}\rangle .
		\end{align*}
		Moreover, by repeated use of Lemma \ref{lem:id} we obtain the following identities:
		\begin{align*}
			\langle v^+_{1,3}\mid v^-_{2,3}\rangle\langle v^+_{4,3}\mid v^-_{1,3}\rangle&=\langle v^+_{1,3}\mid v^-_{1,3}\rangle\langle  v^+_{4,3}\mid v^-_{2,3}\rangle,\\
			\langle v^+_{2,4}\mid v^-_{1,4}\rangle\langle v^+_{1,4}\mid v^-_{3,4}\rangle&=\langle v^+_{1,4}\mid v^-_{1,4}\rangle\langle  v^+_{2,4}\mid v^-_{3,4}\rangle,\\
			\langle v^+_{3,2}\mid v^-_{1,2}\rangle\langle v^+_{1,2}\mid v^-_{4,2}\rangle&=\langle v^+_{1,2}\mid v^-_{1,2}\rangle\langle  v^+_{3,2}\mid v^-_{4,2}\rangle.
		\end{align*}
		Plugging it in we obtain that
		\[\theta_{1,2,3,4}\theta_{1,3,4,2}\theta_{1,4,2,3}=\frac{\langle  v^+_{4,3}\mid v^-_{2,3}\rangle\langle v^+_{2,4}\mid v^-_{3,4}\rangle\langle  v^+_{3,2}\mid v^-_{4,2}\rangle}{\langle  v^+_{2,4}\mid v^-_{2,3}\rangle\langle  v^+_{3,2}\mid v^-_{3,4}\rangle\langle v^+_{4,3}\mid v^-_{4,2}\rangle}.\]
		As $v^+_{i,j}=v^-_{j,i}$ for $n$ even, we conclude that $\theta_{1,2,3,4}\theta_{1,3,4,2}\theta_{1,4,2,3}=1$.
	\end{proof}
	\begin{remark}
		Suppose $g\in\G$ is such that its action on the space of $(n-1)$ dimensional isotropic subspaces has an attracting fixed point $W_a$ and a repelling fixed point $W_r$ and suppose $W_a^\perp$ and $W_r^\perp$ are transverse to each other. We call such elements \emph{proto-pseudo-hyperbolic} and recall that 
		\[g^{\pm1}v^+_{a,r}=\lambda(g)^{\pm1}v^+_{a,r}.\]
	\end{remark} 
	
	\begin{proposition}\label{prop:lambdabeta}
		Suppose $g\in\G$ is a proto-pseudo-hyperbolic element with attracting fixed point $W_a$ and a repelling fixed point $W_r$. Then for any $(n-1)$ dimensional isotropic subspace $W_*$ whose orthogonal space is transverse to both $W_a^\perp$ and $W_r^\perp$ the following holds:
		\begin{enumerate}
			\item $\theta(W_r,W_a,g W_*, W_*)=\lambda(g)^2$ when $n$ is even,
			\item $\theta(W_r,W_a,g W_*, W_*)=1$ when $n$ is odd.
		\end{enumerate}
	\end{proposition}
	\begin{proof}
		We denote $g A_*$ by $A_{g*}$ and use Lemma \ref{lem:id} to deduce that
		\begin{align*}
			\theta(A_r,A_a,g A_*, A_*)&=\frac{\langle v^+_{r,g*}\mid v^-_{a,g*}\rangle\langle v^+_{a,*}\mid v^-_{r,*}\rangle}{\langle v^+_{a,*}\mid v^-_{a,g*}\rangle\langle v^+_{r,g*}\mid v^-_{r,*}\rangle}=\frac{\langle gv^+_{r,*}\mid gv^-_{a,*}\rangle\langle v^+_{a,*}\mid v^-_{r,*}\rangle}{\langle v^+_{a,*}\mid gv^-_{a,*}\rangle\langle gv^+_{r,*}\mid v^-_{r,*}\rangle}\\
			&=\frac{\langle v^+_{r,*}\mid v^-_{a,*}\rangle\langle v^+_{a,*}\mid v^-_{r,*}\rangle}{\langle v^+_{a,*}\mid gv^-_{a,*}\rangle\langle gv^+_{r,*}\mid v^-_{r,*}\rangle}=\frac{\langle v^+_{a,*}\mid v^-_{a,*}\rangle\langle v^+_{r,*}\mid v^-_{r,*}\rangle}{\langle v^+_{a,*}\mid gv^-_{a,*}\rangle\langle gv^+_{r,*}\mid v^-_{r,*}\rangle}.
		\end{align*}
		Again using Lemma \ref{lem:id} twice more we obtain the following two identities:
		\begin{align*}
			\langle v^+_{r,g*}\mid v^-_{r,*}\rangle\langle v^+_{r,a}\mid v^-_{r,a}\rangle&=\langle v^+_{r,g*}\mid v^-_{r,a}\rangle\langle v^-_{r,*}\mid v^+_{r,a}\rangle,\\
			\langle v^-_{a,g*}\mid v^+_{a,*}\rangle\langle v^-_{a,r}\mid v^+_{a,r}\rangle&=\langle v^-_{a,g*}\mid v^+_{a,r}\rangle\langle v^+_{a,*}\mid v^-_{a,r}\rangle.
		\end{align*}
		Therefore, we deduce that:
		\begin{align*}
			\langle gv^+_{r,*}\mid v^-_{r,*}\rangle\langle v^+_{r,a}\mid v^-_{r,a}\rangle&=\langle gv^+_{r,*}\mid v^-_{r,a}\rangle\langle v^-_{r,*}\mid v^+_{r,a}\rangle=\langle v^+_{r,*}\mid g^{-1}v^-_{r,a}\rangle\langle v^-_{r,*}\mid v^+_{r,a}\rangle\\
			&=\lambda(g)^{(-1)^{n-1}}\langle v^+_{r,*}\mid v^-_{r,a}\rangle\langle v^-_{r,*}\mid v^+_{r,a}\rangle\\
			&=\lambda(g)^{(-1)^{n-1}}\langle v^+_{r,*}\mid v^-_{r,*}\rangle\langle v^-_{r,a}\mid v^+_{r,a}\rangle,\\
			\langle gv^-_{a,*}\mid v^+_{a,*}\rangle\langle v^-_{a,r}\mid v^+_{a,r}\rangle&=\langle gv^-_{a,*}\mid v^+_{a,r}\rangle\langle v^+_{a,*}\mid v^-_{a,r}\rangle\\
			&=\langle v^-_{a,*}\mid g^{-1}v^+_{a,r}\rangle\langle v^+_{a,*}\mid v^-_{a,r}\rangle\\
			&=\lambda(g)^{-1}\langle v^-_{a,*}\mid v^+_{a,r}\rangle\langle v^+_{a,*}\mid v^-_{a,r}\rangle\\
			&=\lambda(g)^{-1}\langle v^-_{a,*}\mid v^+_{a,*}\rangle\langle v^+_{a,r}\mid v^-_{a,r}\rangle.
		\end{align*}
		Hence, $\theta(A_r,A_a,g A_*, A_*)=\lambda(g)\lambda(g)^{(-1)^{n}}$ and our result follows.
	\end{proof}
	\begin{remark} \label{rem:conjana}
		Suppose $\rho\in\sHom(\Gamma,\G,\sP^\pm)$. Hence, for all infinite order elements $\gamma\in\Gamma$ we obtain that the action of $\rho(\gamma)$ on the space of $(n-1)$-dimensional isotropic subspaces of $\R^{n,n}$ has an attracting fixed point and a repelling fixed point. We abuse notation and denote the attracting fixed point by $\xi_\rho(\gamma^+)$ and the repelling fixed point by $\xi_\rho(\gamma^-)$. Henceforth, we fix $\rho$ and omit the subscripts $\rho$ from the eigenvalues and crossratios. Also, when there is no confusion of notation, for $a,b,c,d\in\bdry$ all distinct, we denote $\theta(\xi(a), \xi(b), \xi(c), \xi(d))$ by $\theta(a,b,c,d)$.
	\end{remark}
	\begin{proposition}
		Suppose $n$ is even, $\rho\in\sHom(\Gamma,\G,\sP^\pm)$ and $\gamma,\eta\in\Gamma$ are two infinite order elements such that the four points $\gamma^\pm,\eta^\pm\in\bdry$ are distinct and the sequence $\{\gamma^m\eta^k\}_{m\in\mathbb{N}}\subset\Gamma$ contains a subsequence $\{\gamma^{n_i}\eta^k\}_{i\in\mathbb{N}}$ consisting only of infinite order elements. Then the following identity holds:
		\[\lim_{i\to\infty}\frac{\lambda(\gamma^{n_i}\eta^k)^2}{\lambda(\gamma^{n_i})^2\lambda(\eta^k)^2}=\theta(\eta^-, \gamma^-,\gamma^+,\eta^k\gamma^+)\theta(\eta^+,\gamma^+,\gamma^-,\eta^{-k}\gamma^-).\]
	\end{proposition}
	\begin{proof}
		The proof follows exactly word to word as in the proof of Proposition \ref{prop:lim1} by replacing the appearances of $\alpha$ by $\log\lambda$, $\beta$ by $\log\theta$ and replacing the appearances of Proposition \ref{prop:cr} and Proposition \ref{prop:alphabeta} respectively by Proposition \ref{prop:ecr} and Proposition \ref{prop:lambdabeta}.
	\end{proof}
	\begin{proposition}\label{prop:thetalimit}
		Suppose $n$ is even, $\rho\in\sHom(\Gamma,\G,\sP^\pm)$ and $\gamma,\eta\in\Gamma$ are two infinite order elements such that the four points $\gamma^\pm,\eta^\pm\in\bdry$ are distinct and the sequence $\{\gamma^m\eta^m\}_{m\in\mathbb{N}}\subset\Gamma$ contains a subsequence $\{\gamma^{n_i}\eta^{n_i}\}_{i\in\mathbb{N}}$ consisting only of infinite order elements. Then the following identity holds:
		\[\lim_{i\to\infty}\frac{\lambda(\gamma^{n_i}\eta^{n_i})^2}{\lambda(\gamma^{n_i})^2\lambda(\eta^{n_i})^2}= \theta(\eta^-, \gamma^-,\gamma^+,\eta^+)^2.\]
	\end{proposition}
	\begin{proof}
		The proof follows exactly word to word as in the proof of Proposition \ref{prop:lim2} by replacing the appearances of $\alpha$ by $\log\lambda$, $\beta$ by $\log\theta$ and replacing the appearances of Proposition \ref{prop:cr} and Proposition \ref{prop:alphabeta} respectively by Proposition \ref{prop:ecr} and Proposition \ref{prop:lambdabeta}.
	\end{proof}

	\bibliography{Library.bib}
	\bibliographystyle{alpha}
	
\end{document}